\documentclass{article}
\usepackage{fullpage,amssymb,amsmath,amsthm,amsfonts,algorithm,algorithmic,graphicx, subfigure,color,enumitem}
\usepackage{comment}

\usepackage[mathscr]{euscript}
\usepackage[normalem]{ulem} 

\newcommand\numberthis{\addtocounter{equation}{1}\tag{\theequation}}

\newcommand{\bmat}{\left[ \begin{array}}
\newcommand{\emat}{\end{array} \right]}

\newcommand{\ignore}[1]{}
\newtheorem{theorem}{Theorem}[section]

\newtheorem{proposition}[theorem]{Proposition}
\newtheorem{lemma}[theorem]{Lemma}
\newtheorem{corollary}[theorem]{Corollary}

\theoremstyle{definition}
\newtheorem{definition}[theorem]{Definition}

\newtheorem{remark}[theorem]{Remark}

\usepackage{epstopdf}
\DeclareGraphicsExtensions{.eps}
\graphicspath{{figures/}}

\title{An improved analysis and unified perspective on deterministic and randomized low rank matrix approximations}
\author{James Demmel \thanks{Computer Science Division and Department of Mathematics, UC Berkeley,  CA 94720-1776, {demmel@cs.berkeley.edu}} \and Laura Grigori \thanks{Alpines, Inria, Sorbonne Université, Université de Paris, CNRS, Laboratoire Jacques-Louis Lions, F-75012 Paris.  {laura.grigori@inria.fr}. This author has received funding from the European Research Council (ERC) under the European Union’s Horizon 2020 research and innovation programme (grant agreement No 810367). The work was also supported by the NLAFET project as part of European Union’s Horizon 2020 research and innovation program under grant 671633.} \and Alexander Rusciano \thanks{Department of Mathematics, UC Berkeley,  CA 94720-1776, {rusciano@math.berkeley.edu}}}

\begin{document}

\maketitle

\begin{abstract}
We introduce a Generalized LU-Factorization (\textbf{GLU}) for low-rank matrix approximation.  We relate this to past approaches and extensively analyze its approximation properties.  The established deterministic guarantees are combined with sketching ensembles satisfying Johnson-Lindenstrauss properties to present complete bounds.  Particularly good performance is shown for the sub-sampled randomized Hadamard transform (SRHT) ensemble.  Moreover, the factorization is shown to unify and generalize many past algorithms.  It also helps to explain the effect of sketching on the growth factor during Gaussian Elimination.
\end{abstract}

\section{Introduction} \label{sec:intro} 

Many different problem domains produce matrices that can be approximated by a low-rank matrix.  In some cases such as a divide-and-conquer approach to eigenproblems \cite{BDG}, there may be many large and small singular values separated by a gap.  In other cases such as identifying a low rank subspace from noisy data, we might expect there to be relatively few large singular values.  Perhaps most generically in applied problems, there is no pronounced gap, but the spectrum still decays fairly quickly, and one might prefer to work with a more compact representation when computing quantities such as matrix-vector products.

\vspace*{.3cm}
We next define some related properties which can be of interest to these problems.  The following definitions have appeared in the rank-revealing literature, such as in \cite{MG, G, DDH, GE} in similar forms. Here and later the singular values are sorted in descending order.

\begin{definition}\label{def:low_rank}[low-rank approximation]
A matrix $A_k$ satisfying $\|A-A_k\|_2 \leq \gamma \sigma_{k+1}(A)$ for some $\gamma \geq 1$ will be said to be a $(k, \gamma)$ \textit{low-rank approximation of} $A$.
\end{definition}

\begin{definition}\label{def:spectral_approx}[spectrum preserving]
If $A_k$ satisfies $\sigma_j(A) \geq \sigma_{j}(A_k) \geq \gamma^{-1} \sigma_j(A)$ for $j \leq k$ and some $\gamma \geq 1$, it is $(k,\gamma)$ \textit{spectrum preserving}.
\end{definition}

Many results in the rank-revealing literature use a strengthening of Definition \ref{def:low_rank},
\begin{definition}\label{def:kernel_approx}[kernel approximation]
If $A_k$ satisfies $\sigma_{j+k}(A) \leq \sigma_j(A-A_k)\leq \gamma \sigma_{k+j}(A)$ for $1 \leq j \leq n-k$ and some $\gamma \geq 1$, it is a $(k, \gamma)$ \textit{kernel approximation of} $A$.
\end{definition}
In all of these definitions, if we assume $A_k$ is rank $k$, then $\gamma = 1$ is optimal from the truncated-SVD, so all methods can be compared with this standard. Though we made the above definitions quite strong, we will not prove our results satisfy them exactly.  In particular, we drop the upper bound in Definition \ref{def:spectral_approx} and the lower bound in Definition \ref{def:kernel_approx} from our considerations.  One can derive analogs for these dropped quantities using techniques developed in this paper.  For example, one could replace $\sigma_j(A)$ with $\delta \cdot \sigma_j(A)$ in Definition \ref{def:spectral_approx} and it is not generally difficult to give a better bound on $\delta$ than on $\gamma$ by using Definition \ref{def:low_rank} and Weyl's inequality. However, the stated complementary bounds in Defs. \ref{def:spectral_approx} and \ref{def:kernel_approx} do not hold for all the algorithmic variations we consider, and we choose not to complicate the results with these considerations.

\vspace*{.3cm}
Different algorithms may end up representing $A_k$ in different ways, but generally $A_k$ is represented as a product of matrices which have at least one dimension much smaller than those of the original $A$.  Note in this work we do not require $A_k$ to be rank $k$.  Nevertheless, the dimensions of $A_k$ will be chosen as a function of $k$ in order to compete with the truncated SVD of rank $k$, and this motivates the choice of notation.  For the choices made in this paper, it is always the case that $\text{rank}(A_k)= O\left(k \cdot \text{polylog}(n)\right)$. 

\vspace*{.3cm}
This paper has two main goals, both motivated by the history of low-rank factorizations.  First, we show that many important low-rank factorizations can be viewed as an LU-factorization followed by deleting the Schur-complement.  We call this prototype algorithm $\textbf{GLU}$.  Second, older research into low-rank factorizations bounded more quantities than recent results on randomized algorithms.  In particular, Definitions \ref{def:spectral_approx} and \ref{def:kernel_approx} do not receive much discussion in randomized algorithms.  We will provide bounds on all of Definitions \ref{def:low_rank}, \ref{def:spectral_approx}, \ref{def:kernel_approx} for $\textbf{GLU}$. In doing this, we first derive sharp deterministic bounds for approximate LU and QR factorizations in Sections \ref{sec:GLU} and \ref{sec:QR}, and then in section \ref{sec:application} we complete the bounds by using properties of random matrix ensembles.

\vspace*{.3cm}
$\textbf{GLU}$ is essentially an LU-factorization that allows the leading block to be rectangular instead of square.  Allowing the leading block to be rectangular enables much better low-rank approximation properties.  Let $A$ be an $m\times n$ matrix, $A_{11}$ be the leading $l' \times l$ block which is assumed to have full column rank so that $l' \geq l$, and $U$ and $V$ be invertible matrices.  First we have an exact factorization of matrix $A$ that is the natural generalization of a full LU-factorization,

\[
A =
\begin{pmatrix}
A_{11} & A_{12} \\
A_{21} & A_{22}
\end{pmatrix}
=
\begin{pmatrix}
 I &  \\
A_{21} A_{11}^{+} & I
\end{pmatrix}
\begin{pmatrix}
A_{11} & A_{12} \\
     & \mathscr{S}(A_{11})
\end{pmatrix},
\]
where $\mathscr{S}(A_{11})=A_{22}-A_{21}A_{11}^+A_{12}$ denotes what we call the generalized Schur-complement.  By applying the sketching matrices $U$ and $V$ and deleting the Schur complement, we get a low-rank factorization that can have remarkably good properties.  Defining $\bar{A} = UAV$,
\begin{equation}\label{eq:GLUintro}
A_k := 
U^{-1}
\begin{pmatrix}
I \\
\bar{A}_{21} \bar{A}_{11}^{+}
\end{pmatrix}
\begin{pmatrix}
\bar{A}_{11} & \bar{A}_{12}
\end{pmatrix} V^{-1}
\end{equation}
is a complete description of our proposed $\textbf{GLU}$ approximation.  The inverses may look daunting at first because they are large matrices, but we will see that they are only tools to facilitate the analysis; actually the leading rows of $U$ and leading columns of $V$ are the only parts required.

\vspace*{.3cm}
We have emphasized that $\textbf{GLU}$ factorization unifies many factorizations through appropriate choices of the settings of $U$,$V$. We believe it is also important that other choices are novel and practical, as we illustrate in main results Theorem \ref{thm:LU} and Theorem \ref{thm:srht}.  That said, this paper will not argue that these novel instantiations of $\textbf{GLU}$ should necessarily be adopted over similar methods like the low-rank factorization described in \cite{CW}.  On the contrary, we find that the factorization underlying \cite{CW}, which we term \textbf{CW}, can be viewed as an abridged version of $\textbf{GLU}$.  Thus while our bounds are tighter in the case of Definition \ref{def:kernel_approx}, we briefly sketch in Remark \ref{rem:comparison} that the improved bounds on Definitions \ref{def:spectral_approx} and \ref{def:kernel_approx} under the SRHT ensemble as in Theorem \ref{thm:srht} also apply to $\textbf{CW}$.

\vspace*{.3cm}
The remainder of the introduction is divided into four sections for clarity.  The first and second aim to highlight our contributions.  The third provides references to related work.  The fourth gives notation we adopt.

\subsection{Unifying Approach}
$\textbf{GLU}$ generalizes past low-rank LU factorizations in two ways.  First, it allows pre- and post-multiplication by
matrices other than permutations.  Second, it allows for rectangular Schur complements.  Even without generalizing to rectangular Schur complements, GLU encompasses several well-known
procedures.  We provide examples to illustrate this in section
\ref{sec:relationship}.  Table \ref{tab:equivIntro} summarizes several
deterministic and randomized approximation algorithms.  It displays
separately the case when $k \leq l=l'$ and the more general case when
$k \leq l \leq l'$, and cites existing as well as new bounds on the
spectral and kernel approximation provided by these algorithms.  We
discuss a novel and practical instance when $l < l'$ in section \ref{subsec:intronewbounds}.
Here we focus on $k \leq l=l'$ and identify the equivalence between
existing deterministic and randomized algorithms.  In this case, the
rank-k approximation $A_k$ can be written as
\begin{eqnarray}
A_k &=& U^{-1} \left(\begin{array}{c} I_{l} \\ \bar{A}_{21}{\bar{A}^{-1}}_{11} \end{array}\right) \left (\begin{array}{cc} \bar{A}_{11} & \bar{A}_{12} \end{array}\right) V^{-1} \nonumber \\
     &=& A V_1 (U_1 A V_1)^{-1} U_1 A, \label{eq:simpler_looking}
\end{eqnarray}
where $V_1$ contains the leading $l$ columns of $V$, $U_1$ contains the leading $l$ rows of $U$, and $\bar{A} = U A V$.  See \eqref{eq:derivation} for more details.  Now we define some notation we will use later.  Let $Q_1$ be the orthogonal factor obtained from the thin QR-decomposition of $A V_1$, so $Q_1$ is of dimensions $m \times l$. In the case when $U_1$ contains the leading $l$ rows of a permutation matrix $U$, we denote $U Q_1 = \bigl( \begin{smallmatrix} \bar{Q}_{11} \\ \bar{Q}_{21} \end{smallmatrix} \bigr) $, where $\bar{Q}_{11}$ is $l \times l$.  While $l \geq k$ is always the case, in applications $l$ varies from being exactly $k$, as for deterministic algorithms, to being a polylog-factor larger than $k$ for randomized algorithms.

\vspace*{.3cm}

Deterministic algorithms are typically based on rank revealing QR and
LU factorizations.  Both factorizations select $k$ columns from the
matrix $A$, that is $V_1$ represents a column permutation and $A V_1$
are the selected columns.  In the case of a rank revealing QR
factorization, $U_1 = Q_1^T$ and the approximation becomes $A_k = Q_1
Q_1^T A$. See \eqref{eq:transLUQR} for a detailed derivation.  Let
$Q_1^T A = (R_{11} \; R_{12})$.  The strong rank-revealing
$QR$-factorization \cite{GE} chooses the column permutation $V_1$ such
that $|| R_{11}^{-1} R_{12} ||_{max}$ is bounded by a small constant
and the approximation $A_k$ is spectrum preserving and a kernel approximation of
$A$: $\gamma$ in Definitions \ref{def:spectral_approx} and
\ref{def:kernel_approx} is a low degree polynomial in $n$ and $k$.
The rank revealing LU factorization selects $k$ columns and $k$ rows
from the matrix $A$, that is both $U_1$ and $V_1$ are permutation
matrices. For example in \cite{GCD} the columns are selected by using
a pivoting strategy referred to as tournament pivoting and based on
rank revealing QR, while the rows are selected such that $||\bar{Q}_{21}
\bar{Q}_{11}^{-1}||_{max}$ is bounded.  The obtained approximation $A_k = A
V_1 (U_1 A V_1)^{-1} U_1 A$ is again spectrum preserving and a kernel approximation
of $A$, with $\gamma$ in Definitions \ref{def:spectral_approx} and
\ref{def:kernel_approx} being a low degree polynomial in $n$ and $k$.

\vspace*{.3cm} For randomized algorithms, $V_1$ is a random matrix,
typically based on Johnson-Lindenstrauss transforms or fast
Johnson-Lindenstrauss transforms, such as the sub-sampled randomized
Hadamard transform (SRHT) of Definition \ref{def:srht} introduced originally in \cite{S}.  The randomized SVD (see e.g. \cite{HMT})
is obtained by choosing $U_1= Q_1^T$ and corresponds to computing $l$
steps of the QR factorization of $UAV$.  We refer to this
factorization as a randomized QR factorization.  The randomized SVD
via row extraction is obtained by choosing $U_1$ a row permutation
such that $||\bar{Q}_{21} \bar{Q}_{11}^{-1}||_{max}$ is bounded.  In
other words, this factorization corresponds to computing $l$ steps of
the LU factorization of $UAV$, and we refer to this as randomized LU
with row selection.  Notably in the case of the approximations based
on LU factorization, both deterministic and randomized algorithms
bound $||\bar{Q}_{21} \bar{Q}_{11}^{-1}||_{max}$ \cite{GE, GCD} to obtain
guarantees on the approximation.

\begin{table}[h!]
\begin{tabular}{c|c}
    \multicolumn{2}{c}{\textbf{Existing algorithms: Instances of $V_1, U_1$ and the approximation $A_k$ for $k \leq l = l'$,} } \\
    \multicolumn{2}{c}{$A_k = A V_1 (U_1 A V_1)^{-1} U_1 A$} \\
    \textbf{Deterministic algorithms and bounds} & \textbf{Randomized algorithm and bounds} \\ \hline
    \textbf{QR with column selection}, $k=l=l'$ & \textbf{Randomized QR}, $k < l = l'$  \\ 
    {\footnotesize (a.k.a. strong rank revealing QR, \cite{GE})} & 
    {\footnotesize (a.k.a. randomized SVD, e.g. \cite{HMT})} \\
    {\small $V_1$  permutation, $U_1 = Q_1^T$, $A_k = Q_1 Q_1^T A$, see \eqref{eq:transLUQR}}
    &
    {\small $V_1$ random, $U_1 = Q_1^T$, $A_k = Q_1 Q_1^T A$, see \eqref{eq:transLUQR}} \\ 
    {\small \textit{new bounds on kernel approx. (Proposition \ref{prop:QR})} }
    & \textit{new bounds on kernel approx (Cor. \ref{cor:srht_qr})} \\
    \hline 
    \textbf{LU with column/row selection}, $k=l=l'$  & \textbf{Randomized LU with row selection}, $k < l = l'$ \\
    { \footnotesize (a.k.a. rank revealing LU)}  & { \footnotesize (a.k.a. randomized SVD via Row extraction, , e.g. \cite{HMT})}  \\
    {\small $V_1$, $U_1$ permutations} &  {\small $V_1$ random, $U_1$ permutation, see \eqref{eq:transLUQR}} \\ 
    {\small \textit{new spectral, kernel bounds (Proposition \ref{prop:LU})} }
    &  \textit{new bounds possible (Thm \ref{thm:LU}, Cor. \ref{cor:srht_qr})} \\
    \hline \hline
    \multicolumn{2}{c}{} \\
    \multicolumn{2}{c}{\textbf{$V_1, U_1$ and the approximation $A_k$ for $k \leq l \leq l'$,} } \\
    \multicolumn{2}{c}{$A_k = [U_1^+ (I - (U_1 A V_1)(U_1 A V_1)^+) + (A V_1)(U_1 A V_1)^+ ] [ U_1 A]$, see \eqref{eq:simplified_rectangular}} \\ 
    \textbf{Deterministic algorithm and bounds} & \textbf{Randomized algorithm and bounds} \\ \hline
    \textbf{LU with column/row selection}, $k \leq l \leq l'$ & \textbf{Randomized LU}, $k \leq l \leq l'$  \\
    {\small $V_1$, $U_1$ permutations} & {\small $V_1, U_1$ random } \\ 
    {\small \textit{new spectral, kernel bounds (Proposition \ref{prop:LU})} }
    &
    {\small \textit{new spectral, kernel bounds (Theorems \ref{thm:LU}, \ref{thm:srht})} } \\
\end{tabular}
\caption{Summary of several deterministic and randomized algorithms
  for computing $A_k$, the low rank approximation of a matrix $A$ of
  dimensions $m \times n$. $U_1$ is $l' \times m$, $V_1$ is $n \times
  l$, and $Q_1$ is the $m \times l$ orthogonal factor obtained from
  the thin QR-decomposition of $A V_1$. In the case when $U_1$
  contains the leading $k$ rows of a permutation matrix $U$, we denote
  $U Q_1 = \left(\bar{Q}_{11}^T \, , \bar{Q}_{21}^T \right)^T $, where $\bar{Q}_{11}$
  is $l \times l$.
  }
\label{tab:equivIntro}
\end{table}

\subsection{Detailed Bounds} \label{subsec:intronewbounds}

In the context we consider, $\textbf{GLU}$ satisfies bounds at least as sharp as in the literature, and many are new.

Given $k \leq l \leq l'$, the clean formulation of $A_k$ described in eq. (\ref{eq:simpler_looking}) becomes a bit more complicated,
\begin{equation}
\label{eq:introGLU}
A_k  =  [U_1^+ (I - (U_1 A V_1)(U_1 A V_1)^+) + (A V_1)(U_1 A V_1)^+ ] [ U_1 A],
\end{equation} 
where $U_1$ and $(U_1 A V_1)$ are of dimensions $l' \times m$ and $l' \times l$ respectively.  However, the algorithmic implementation is still straightforward and inexpensive.  See \eqref{eq:simplified_rectangular} for a detailed derivation.  Proposition \ref{prop:LU} gives the precursor bounds for the spectral and the kernel approximation provided by $A_k$ for general $U_1$ and $V_1$, and as in Section \ref{sec:application} properties of $U_1$ and $V_1$ specific to the algorithm are used to complete the bound.  Both Propositions \ref{prop:LU} and \ref{prop:QR} provide new deterministic bounds not found in the literature.  For example, Proposition \ref{prop:QR} generalizes Theorem 9.1 of \cite{HMT} to include values $j>1$.  This generalization proves useful when analyzing Definition \ref{def:kernel_approx} for randomized algorithms, which we observe to be an advantage of $\textbf{GLU}$ over $\textbf{CW}$.

\vspace*{.3cm}
Section \ref{sec:application} contains our new results after suitable random ensembles are chosen, that is when $V_1$ and $U_1$ are random matrices. Extra attention is given to the SRHT ensemble of Definition \ref{def:srht}, because the especially good bounds it can provide were not fully exploited in past literature.  Using this ensemble, from Algorithm \ref{alg:grlu} for computing $\textbf{GLU}$ we can see the number of arithmetic operations is $O(nm \log(l') + mll')$.  Plugging in $l'$ and $l$ from Theorem \ref{thm:srht}, we can produce a low-rank approximation in $\tilde{O}(nm + k^2m \epsilon^{-3})$ time that relative to the squared error of the truncated SVD of rank $k$, $A_{\text{opt},k}$,
\begin{itemize}
\item approximates $A$ with only $1+O(\epsilon)$ times the squared Frobenius norm error .
\item approximates $A$ with only $O\left(1+\frac{\log(m/\delta)\epsilon }{k\log(k/\delta)} \frac{\|A-A_{\text{opt},k}\|_F^2}{\|A-A_{\text{opt},k}\|_2^2}\right)$ times the squared spectral norm error.
\end{itemize}
This holds with probability $1-5\delta$, and $l, l'$ grow poly-logarithmically with $\delta$, as in Remark \ref{rem:runtime}. In other words, the algorithm we propose attains $\gamma = O(1)$ in Definition \ref{def:low_rank} for many families of $A$ matrices encountered in practice with modest spectral decay (which makes the Frobenius norm not too much larger than the spectral norm).  The same Theorem \ref{thm:srht} shows this $\gamma = O(1)$ bound carries over to Definition \ref{def:kernel_approx}.  Further, Theorem \ref{thm:LU} shows that Definition \ref{def:spectral_approx} is satisfied with $\gamma = O(\frac{k}{n})$.   To our knowledge, no other work has found such a representation of $A$ in time less than $\Omega(nm k)$ satisfying any of these properties.  Instead, randomized low-rank approximation literature on algorithms running in $\Omega(nm k)$ time do not typically discuss the spectral norm of the residual (Definition \ref{def:low_rank}), choosing to focus on the Frobenius norm.  Moreover, the fast linear algebra community has typically not considered properties like spectrum preserving and kernel approximation (Definitions \ref{def:spectral_approx}, \ref{def:kernel_approx}).

\vspace*{.3cm}
Our bounds have interesting implications for the growth factor of pre/post-conditioned Gaussian Elimination.  Corollary \ref{cor:provable} is a step towards a theoretical understanding of conditioning Gaussian Elimination to avoid pivoting.  Besides this, it expands the classes of distributions for which pivoting is provably unnecessary, to a class including Gaussian-distributed matrices.  We pose an open question at the end, motivated by this analysis.

\subsection{Related Work}
Low-rank matrix approximations have been extensively studied, hence
this work is related to a large body of literature.  Because of our
emphasis on the LU-factorization viewpoint, we should mention some
work related to LU factorizations.  Such papers providing information
regarding Definitions \ref{def:low_rank}, \ref{def:spectral_approx},
\ref{def:kernel_approx} are few, notably including perhaps the first
\cite{MG}, as well as later more efficient versions like \cite{GCD}.
These papers do not exploit randomness, however.

\vspace*{.3cm}
Exploiting randomness for low-rank factorizations has led to major speedups. Some literature in recent years has exploited this for LU factorizations, including perhaps most relevantly \cite{SSAA}.  Their work has somewhat different goals, in that it seeks to find left and right permutation matrices, which makes it in some ways more like \cite{GCD}. Also, their paper only discusses spectral norm bounds on the residual.  Interestingly, the fast version of their procedure (their Algorithm 4.4) uses an ensemble equivalent to the SRHT ensemble.  The bounds we have in Theorem \ref{thm:srht} are better for the spectral norm of the residual.  Comparing our Theorem \ref{thm:srht} with their Theorem 4.12, our approximation is always a factor on the order of $\sqrt{n}$ more accurate, and a factor $n$ more accurate when the spectrum decays sufficiently quickly.  Our results utilizing the SRHT ensemble build on \cite{BG}, which proved the SRHT ensemble has geometry preserving properties beyond those of the Johnson-Lindenstrauss transform properties.  They used this fact to provide sharper spectral norm bounds on the residual for the randomized QR decomposition approach to low-rank matrix approximation.

\vspace*{.3cm}
Outside of research into LU factorizations, many papers have focused on studying Johnson-Lindenstrauss embeddings.  This has culminated in algorithms considered to run in nnz(A) time for many problems related to and including low-rank approximations.  Notable such papers include \cite{CW} and \cite{NN}.  This body of literature has focused more on the properties of the random ensemble, and little on the properties of the factorization itself.  For example, \cite{NN} uses the same factorization as \cite{CW}, whose technical report we believe to be the first paper to use sketching from the left and right to speed up the algorithm.  Few of these papers for nnz(A) algorithms study any error bounds beyond the Frobenius norm of the residual.

\vspace*{.3cm}
To date, procedures for the residual being within an $\epsilon$ factor as accurrate as the truncated SVD with respect to the spectral norm do not gain any speed advantage by using fast Johnson-Lindenstrauss ensembles.  This is because a repeated-squaring must be used, and therefore structured sketching matrices have no advantage.  Important work in this area includes \cite{G} and \cite{MM}.

\vspace*{.3cm}
The list is far from complete, and many different takes on the problem have been proposed which tangentially touch this paper, \cite{HMT} and \cite{W} are useful for finding more pointers into the literature.

\subsection{Notations}
As this paper is notation heavy, we first take a moment to collect some conventions we will use.
\begin{itemize}
\item $A$ is $m \times n$.
\item $A_{\text{opt},k}$ will be the truncated rank-k SVD.
\item Assume $m \geq n$. $[Q, R] = \textbf{tQR}(A)$ will be the thin QR-decomposition of $A$, so $Q$ is $m \times n$
\item $[Q, R] = \textbf{QR}(A)$ will be the square QR-decomposition of $A$, so $Q$ is $m \times m$.
\item Assume $m \leq n$.  $[L, Q] = \textbf{tLQ}(A)$ will be the thin LQ-decomposition of $A$, so $Q$ is $m \times n$
\item $[L, Q] = \textbf{LQ}(A)$ will be the square LQ-decomposition of $A$, so $Q$ is $n \times n$.
\item $[U, \Sigma, V] = \textbf{tSVD}(A)$ will be the thin variant (square $\Sigma$) and with decreasing singular values.  So given $m \geq n$, $ A = U \Sigma V^T$, singular values are $\sigma_1 \geq \dots, \geq \sigma_n$ and $U$ is $m \times n$.
\item $A^+$ is the $n \times m$ Moore-Penrose pseudo-inverse.
\item $[U, \Sigma, V] = \textbf{SVD}(A)$ will be the full variant ($m \times n$ $\Sigma$) and with decreasing singular values.  So $U$ is $m \times m$, $V$ is $n \times n$.
\item $\mathscr{S}(A_{11}) = A_{22}-A_{21}A_{11}^+A_{12}$ is the Schur complement of $A_{11}$; if the dimension of $A_{11}$ is $\l'\times l$, then $\mathscr{S}(A_{11})$ is $(m-l')\times(n-l)$.  Here $A = \left(\begin{array}{cc} A_{11} & A_{12} \\ A_{21} & A_{22} \end{array} \right)$.
\item Matlab-like notation to select submatrices, e.g. $A[:k,:k]$ is the leading $k \times k$ minor of $A$.
\item To simplify notation, we denote $(X_{11})^{+}$ as $X_{11}^{+}$.
\end{itemize}

\section{Generalized LU-factorization}
\label{sec:GLU}

Classically as in \cite{MG} and \cite{GCD}, the rank-revealing LU factorization finds permutations $P_r, P_c$ (usually iteratively over the procedure), forming $\bar{A} = P_r A P_c$, and LU-factors $\bar{A}$ but deletes the Schur-complement after $k$-steps.  Thus,

\[
\bar{A} = \left(\begin{array}{cc} I & 0 \\ \bar{A}_{21} \bar{A}_{11}^{-1}& I \end{array}\right) \left(\begin{array}{cc} \bar{A}_{11} & \bar{A}_{12} \\ 0 & \mathscr{S}(\bar{A}_{11}) \end{array}\right) \approx \left(\begin{array}{c} I \\ \bar{A}_{21}\bar{A}_{11}^{-1} \end{array}\right) \left (\begin{array}{cc} \bar{A}_{11} & \bar{A}_{12} \end{array}\right) =: \bar{A}_k .
\]

This naturally suggests the approximation $A \approx A_k := P_r^T \bar{A}_k P_c^T$.  Letting $P_{c1}$ be the first $k$ columns of $P_c$ and $P_{r1}$ be the first $k$ rows of $P_r$, some algebra (see Remark \ref{rem:calculation} for the more general case) shows the approximation to be $A \approx A P_{c1} (P_{r1} A P_{c1})^{-1} P_{r1} A $.

\vspace*{.3cm}
This paper generalizes the rank-revealing LU-factorization in two directions.  First, we include other matrices on the left and right besides permutations.  This allows for speedups through matrix sketching.  Second, we generalize one step further by using rectangular Schur complements. This can greatly improve the quality of the low-rank approximation, as we will see in Proposition \ref{prop:LU} and Theorem \ref{thm:LU}.

\vspace*{.3cm}
We describe this second modification in greater detail now.  For the sake of analysis it will be convenient to let $U,V$ be square matrices in the following discussion and subsequent Proposition \ref{prop:LU}.  The relevant matrices are the $m \times n $ matrix $A$ which we wish to approximate, the invertible $m \times m$ matrix $U$, and the invertible $n \times n$ matrix $V$.  Now define

\[
\bar{A} := UAV = \left(\begin{array}{cc} I_{l'} & 0 \\ \bar{A}_{21}\bar{A}^+_{11} & I_{m-l'} \end{array}\right) \left(\begin{array}{cc} \bar{A}_{11} & \bar{A}_{12} \\ 0 & \mathscr{S}(\bar{A}_{11}) \end{array}\right) ,
\]
where this is valid when the $l' \times l$ block $\bar{A}_{11}$ has full column rank so that $\bar{A}_{11}^+ \bar{A}_{11}=I$.  In particular we are assuming $l' \geq l$.  To help visualize the construction, the following depicts the block sizes.

\[
\bar{A} = \left(\begin{array}{cc} l',l & l', n-l \\ m-l', l & m-l',n-l \end{array}\right) = \left(\begin{array}{cc} l',l' &  l' ,m-l'\\ m-l',l' & m-l',m-l' \end{array}\right) \left(\begin{array}{cc} l', l & l',n-l  \\ m-l',l & m-l', n-l\end{array}\right) .
\]

Deleting the $(m-l')\times (n-l)$ Schur complement and undoing the $U, V$ factors gives the approximation we use as a definition,

\begin{equation}\label{eq:derivation}
A \approx A_k := U^{-1} \left(\begin{array}{c} I_{l'} \\ \bar{A}_{21}\bar{A}^+_{11} \end{array}\right) \left (\begin{array}{cc} \bar{A}_{11} & \bar{A}_{12} \end{array}\right) V^{-1} .
\end{equation}

In (\ref{eq:derivation}), $U$ and $V$ are square, but for low-rank approximations this would be expensive. Only the leading $l'$ rows and $l$ columns respectively of $U$ and $V$ respectively are actually required, but we find the square form helpful for the analysis.  Accordingly for $U$, we assume that we may express
\begin{equation}\label{eq:restrict}
U = \left(\begin{array}{c} U_1 \\ U_2 \end{array} \right) = \left(\begin{array}{c} L_{11}' U_1' \\ U_2' \end{array} \right) = \left(\begin{array}{cc} L_{11}' & 0 \\ 0 & I_{m-l'} \end{array} \right) \left( \begin{array}{c} U'_1 \\ U'_2 \end{array}\right)
= L' U' ,
\end{equation}
where $U' = \left(\begin{array}{c} U_1' \\ U_2' \end{array} \right)$ is an orthogonal matrix, $U_1$ and $U_1'$ are $l' \times m$, and $L_{11}$ is $l' \times l'$ lower-triangular. Note by assumption, $L_{21}'$ and $L_{12}'$ are $0$ matrices, and $L_{22}'=I_{m-l'}$.  Conceptually this means the first $l'$ rows of $U$ are arbitrary full-rank and the other rows are the orthogonal complement.  We also assume $L'$ is invertible, so that $U$ is invertible as well. Any reasonable sketching matrix $U_1$ satisfies this property with probability $1$. Similarly, we assume $V$ may be expressed as
\begin{equation}\label{eq:restrict2}
V = \left(\begin{array}{cc} V_1 & V_2 \end{array} \right) = \left(\begin{array}{cc} V_1' R_{11}' & V_2' \end{array}\right) = V' \left(\begin{array}{cc} R_{11}' & 0 \\ 0 & I_{n-l} \end{array} \right) = V' R' ,
\end{equation}
where $V'$ is orthogonal, $V_1$ and $V_1'$ are $n \times l$, and $R_{11}'$ is $l \times l$ upper-triangular.  Again note the assumption $R_{21}'$ and $R_{12}'$ and both 0 matrices, and $R_{22}'=I_{n-l}$.  We will again assume $R'$ is invertible so that $V$ is as well.  That $R'_{22}$ and $L'_{22}$ are identity matrices will be used several times in our algebra, so we emphasize this fact.

\vspace*{.3cm}
Schur complements of rectangular blocks do not appear to be commonly used.  The following derives a few useful identities for them in the context of LU-factorization.
\begin{lemma}\label{lem:basic_facts}
We continue to assume $l'\geq l$ and that $\bar{A}_{11}$ has full column rank so that $\bar{A}_{11}^+ \bar{A}_{11} = I$.  Further introduce matrices $U = L' U'$ and $V = V' R'$ structured as explained in (\ref{eq:restrict}) and (\ref{eq:restrict2}).  Set $[Q,R] = \textbf{QR}(AV)$ so that $R$ is $m \times n$.  Block $R$ so that $R_{11}$ is $l \times l$ and $X:=UQ$ so so that $X_{11}$ is $l' \times l$.  Then the following identities hold for $\bar{A} = U A V$,

\begin{align}
\mathscr{S}(\bar{A}_{11}) = \mathscr{S}(X_{11}) R_{22} \label{eq:lu2} \\
\bar{A}_{11} = X_{11} R_{11} \label{eq:lu3} \\
\bar{A}_{21} \bar{A}_{11}^+ = X_{21}X_{11}^{+}\label{eq:lu1} .
\end{align}
\end{lemma}
\begin{proof}
There is a factorization through a generalized LU-factorization of $\bar{A}$, in which the lower-triangular factor is the identity on the diagonal and the lower left factor is $\bar{A}_{21} \bar{A}_{11}^{+}$,
\begin{equation}
\bar{A} = \left( \begin{array}{cc} I_{l'} & \\ \bar{A}_{21} \bar{A}_{11}^{+} & I_{m-l'} \end{array} \right)\left(\begin{array}{cc} \bar{A}_{11} & \bar{A}_{12} \\ & \mathscr{S}(\bar{A}_{11}) \end{array} \right ) . \label{eq:option1}
\end{equation}
However we could alternatively first use a QR-factorization of $AV$ followed by a generalized LU-factorization of $X$ (so that $X_{11}$ is $l' \times l$), 

\begin{align*}
\bar{A} &= U A V \\
&= X R \\ 
&= \left ( \begin{array}{cc} I_{l'} & \\ X_{21}X_{11}^{+} & I_{m-l'} \end{array} \right) \left(\begin{array}{cc} X_{11} & X_{12} \\ & \mathscr{S}(X_{11}) \end{array} \right ) \left(\begin{array}{cc} R_{11} & R_{12} \\ & R_{22} \end{array} \right) \\
&= \left ( \begin{array}{cc} I_{l'} & \\ X_{21}X_{11}^{+} & I_{m-l'} \end{array} \right) \left(\begin{array}{cc} X_{11}R_{11} & \dots \\ & \mathscr{S}(X_{11})R_{22}\end{array} \right ) . \numberthis \label{eq:option2}
\end{align*}
The proof amounts to equating the blocks now between (\ref{eq:option1}) and (\ref{eq:option2}), but we provide a justification which essentially argues that the lower left block of the generalized LU factorization makes it unique.  (\ref{eq:lu3}) follows first because $I_{l'} X_{11} R_{11} = \bar{A}_{11}$.  

\vspace*{.3cm}
Second, by definition $\left(\begin{array}{c} \bar{A}_{11} \\ \bar{A}_{21} \end{array}\right) = \left(\begin{array}{c} X_{11}R_{11} \\ X_{21}R_{11} \end{array}\right)$.  We assumed $\bar{A}_{11}$ has full column rank and we continually assume $U,V$ are invertible; therefore $X_{11}$ has full column rank and $R_{11}$ has full row rank (it is invertible).  Consequently we may compute the pseudo-inverse $\bar{A}_{21} \bar{A}_{11}^+ = X_{21} R_{11} (X_{11} R_{11})^+ = X_{21} X_{11}^+$.  This gives (\ref{eq:lu1}).

\vspace*{.3cm}
 Finally, (\ref{eq:lu2}) follows by equating the corresponding lower-right block of the upper-triangular factors, and is justified because we have shown the left-triangular factors in (\ref{eq:option1}) and (\ref{eq:option2}) are identical (and invertible).
\end{proof}

Singular values of a matrix product obey a well-known bound called the multiplicative Weyl inequality.  We make use of this and its less known reverse version.  Therefore we state the inequality with a reference, and prove its reverse version.

\begin{lemma}\label{lem:interlace}
Say $A$ is $m \times n$, $B$ is $n \times p$.  For $1 \leq k \leq j$, 
\begin{equation}\label{eq:weyl1}
\sigma_j(AB) \leq \sigma_{j-k+1}(A) \sigma_{k}(B) .
\end{equation}
Now assume for simplicity that $n \geq m \geq p$, both $A,B$ are full rank, and $\text{im}(B) \subset \text{ker}(A)_{\perp}$.  In other words, $A$ is short-wide and $B$ is tall-skinny, and the image of $B$ is orthogonal to the kernel of $A$.  Then for $1 \leq k \leq m-j$ and $j \leq p$, an inequality in the other direction is
\begin{equation}\label{eq:weyl2}
\sigma_{m-k+1}(A) \sigma_{j+k-1}(B) \leq \sigma_j(AB) .
\end{equation}
Besides these multiplicative inequalities, the additive Weyl inequality holds for any matrices $A,B$ and $1 \leq k, j \leq n$ where $n$ is the smaller of the row and column numbers, and says
\begin{equation}\label{eq:weyl_add}
\sigma_{j}(A+B) \leq \sigma_{j-k+1}(A)+ \sigma_k(B) .
\end{equation}
\end{lemma}

\begin{proof}
(\ref{eq:weyl1}) and ($\ref{eq:weyl_add}$) are well-known.  For example, section 7.3, exercise 18 from \cite{HJ}.

\vspace*{.3cm}
We next prove (\ref{eq:weyl2}).  Let $\Sigma_1, \Sigma_2$ be the square singular value matrices of $A$,$B$ respectively.  Then $AB$ is spectrally equivalent to $\Sigma_1 U \Sigma_2$ for some $m \times p$ orthogonal matrix $U=V_1^TU_2$, with $U_2$ being the left singular matrix of $B$ and $V_1$ being the right singular matrix of $A$.  This $U$ has orthonormal columns because it is norm preserving; $\text{im}(U_2) \subset \text{im}(V_1) = \text{ker}(V_1^T)_{\perp}$ so if we let $V$ extend $V_1$ to a square orthogonal matrix, then $\|V_1^T U_2 x\|_2 = \|V^T U_2 x\|_2 = \|x\|_2$.  $\Sigma_1$ is invertible based on the full rank assumption, and $U \Sigma_2$ is $m \times p$ with full column rank.  Note that $(U \Sigma_2)^+ \Sigma_1^{-1}$ is a left inverse for $\Sigma_1 U \Sigma_2$.  Therefore $(\Sigma_1 U \Sigma_2)^+ = (U \Sigma_2)^+ \Sigma_1^{-1} P $ where $P$ orthogonally projects onto $\text{im}(\Sigma_1 U \Sigma_2)$.  Apply (\ref{eq:weyl1}) to conclude $\sigma_j((\Sigma_1 U \Sigma_2)^+) \leq \sigma_j((U \Sigma_2)^+ \Sigma_1^{-1})$.  Combine this with another application of (\ref{eq:weyl1}),
\begin{align*}
\sigma_{p-j+1}^{-1}(AB) = \sigma_j((AB)^+) \leq \sigma_j((U \Sigma_2)^+ \Sigma_1^{-1}) &\leq \sigma_{j-k+1}((U \Sigma_2)^+) \sigma_{k}(\Sigma_1^{-1})  \\
 &= \sigma_{p-(j-k+1)+1}^{-1}(V_1^TU_2\Sigma_2) \sigma_{m-k+1}^{-1}(A) \\
 &= \sigma_{p-(j-k+1)+1}^{-1}(\Sigma_2) \sigma_{m-k+1}^{-1}(A) \numberthis \label{eq:this_one} \\
 &= \sigma_{p-(j-k+1)+1}^{-1}(B) \sigma_{m-k+1}^{-1}(A)  .
\end{align*}
We used that $V_1^T U_2$ is an orthogonal matrix to advance to line (\ref{eq:this_one}).  Finally, reassign $j = p-j+1$ to get the claimed (\ref{eq:weyl2}).

\end{proof}

The next Proposition is critical for understanding the rank-revealing properties for $\textbf{GLU}$.  It will combine with Proposition \ref{prop:QR} to culminate in Theorem \ref{thm:LU}.

\begin{proposition}\label{prop:LU}
Let $A$ be an $m \times n$ matrix, $U=L'U'$ and $V=V'R'$ as in (\ref{eq:restrict}) and (\ref{eq:restrict2}), $[Q,R] = \textbf{QR}(AV)$, and finally $\bar{A} = UAV$.  Block $Q, R, A, \bar{A}$ as in Lemma \ref{lem:basic_facts}; in particular $Q_{11}$ is $l' \times l$ and $R_{11}$ is $l\times l$.  Then the low-rank approximation suggested in (\ref{eq:derivation}), namely

\[
A_k := U^{-1} \left(\begin{array}{c} I \\ \bar{A}_{21}\bar{A}^+_{11} \end{array}\right) \left (\begin{array}{cc} \bar{A}_{11} & \bar{A}_{12} \end{array}\right) V^{-1} ,
\]
 satisfies 
 \begin{align}
 \|A-A_k\|_F^2 &\leq \|R_{22}\|_F^2 + \|(UQ)_{11}^+ (UQ)_{12} R_{22}\|_F^2 \label{eq:lu_frob} \\
 \|(A-A_k)-(A-A_k)_{\text{opt},j-1}\|_F^2 &\leq \|R_{22}-{R_{22}}_{\text{opt},j-1}\|_F^2 + \|(UQ)_{11}^+ (UQ)_{12} (R_{22}-{R_{22}}_{\text{opt},j-1})\|_F^2 \label{eq:lu_frob2} \\
 \|A-A_k\|_2^2 &\leq \|R_{22}\|_2^2 + \|(UQ)_{11}^+ (UQ)_{12} R_{22}\|_2^2 \label{eq:lu_spectral}\\
\sigma_j^2(A-A_k) &\leq \|R_{22}-{R_{22}}_{\text{opt},j-1}\|_2^2 + \|(UQ)_{11}^+ (UQ)_{12}(R_{22}-{R_{22}}_{\text{opt},j-1})\|_2^2 \label{eq:lu_spectral2} \\
\sigma_i(A_k) &\geq  \sigma_i(A_k[:,:l']) = \sigma_{i}(R_{11} R_{11}'^{-1}). \label{eq:assumption}
\end{align} 
 
In the above, the relations for $\sigma_j$ hold for $1 \leq j \leq \min(m,n)-k$.  The relation for $\sigma_i$ holds for $1\leq i \leq k$.  Also note that $\sigma_{j}(R_{11} R_{11}'^{-1}) $ could be thought of as the singular values of $A$ restricted to $\text{im}(A V_1)$.
\end{proposition}

\begin{proof}
The approximation loss in $A_k$ is exactly the Schur complement $\mathscr{S}(\bar{A}_{11})$.  To establish this, we first do some matrix algebra.  In this algebra we will again recall the simplifying notation $X := UQ$ from Lemma \ref{lem:basic_facts}. Now to start, we have
\begin{align*}
A_k &= U^{-1} \left(\begin{array}{c} I \\ \bar{A}_{21}\bar{A}^+_{11} \end{array}\right) \left(\begin{array}{cc}\bar{A}_{11} & \bar{A}_{12} \end{array} \right) V^{-1} \\
&= U^{-1} \left(\begin{array}{cc} \bar{A}_{11} & \bar{A}_{12} \\ \bar{A}_{21} & \bar{A}_{21} \bar{A}_{11}^{+} \bar{A}_{12} \end{array}\right) V^{-1}. \numberthis \label{eq:for_later}
\end{align*}
Next apply (\ref{eq:lu2}) from Lemma \ref{lem:basic_facts} to get $\mathscr{S}(\bar{A}_{11}) = \mathscr{S}(X_{11}) R_{22}$. From this and the fact that $U^{-1} \bar{A} V^{-1} = A$,
\begin{align*}
A-A_k &= U^{-1} \left( \begin{array}{cc} 0 & \\ & \mathscr{S}(\bar{A}_{11}) \end{array}\right) V^{-1} \\
&= U'^{-1} L'^{-1} \left( \begin{array}{cc} 0 & \\ & \mathscr{S}(X_{11})R_{22}\end{array} \right) R'^{-1} V'^{-1} \\
&= U'^T \left( \begin{array}{cc} 0 & \\ & \mathscr{S}(X_{11})R_{22} \end{array} \right)V'^T \numberthis \label{eq:to_compare} .
\end{align*}
Now to get (\ref{eq:lu_frob}), recalling $U = L'U' $ with $L'_{22}=I$, 
\begin{align*}
 \|A_k-A\|_F^2 &= \|\mathscr{S}(X_{11}) R_{22}\|_F^2 \\
  &= \|\left[(UQ)_{22}-(UQ)_{21}((UQ)_{11})^+ (UQ)_{12} \right]R_{22}\|_F^2 \\
  &= \|\left(\begin{array}{cc}(U'Q)_{21} & (U'Q)_{22} \end{array}\right)\left(\begin{array}{c} -((UQ)_{11})^+ ((UQ)_{12}) R_{22} \\ R_{22}\end{array} \right)\|_F^2\\
  &\leq \|R_{22}\|_F^2 + \|X_{11}^+ X_{12} R_{22}\|_F^2 .
\end{align*}
And for (\ref{eq:lu_spectral}), similar steps produce 
\[
\|A_k-A\|_2^2 \leq \|\left(\begin{array}{c}X_{11}^+ X_{12}R_{22} \\ R_{22}\end{array}\right)\|_2^2 \leq \|X_{11}^+ X_{12}R_{22}\|_2^2 + \|R_{22}\|_2^2 .
\]
Even more generally, from the multiplicative Weyl inequality, $\sigma_j(A_k-A) \leq \sigma_j(\left(\begin{array}{c} -((UQ)_{11})^+ ((UQ)_{12}) R_{22} \\ R_{22}\end{array} \right))$.  Using this, and the additive Weyl inequality \cite{HJ} in the second inequality,
\begin{align*}
\sigma_{j+s-1}^2(A-A_k) &\leq \sigma_{j+s-1}^2(\left(\begin{array}{c} -((UQ)_{11})^+ ((UQ)_{12}) (R_{22}-{R_{22}}_{\text{opt},j-1}+{R_{22}}_{\text{opt},j-1}) \\ R_{22}-{R_{22}}_{\text{opt},j-1}+{R_{22}}_{\text{opt},j-1} \end{array}\right)) \\ 
&\leq \sigma_{s}^2(\left(\begin{array}{c} -((UQ)_{11})^+ ((UQ)_{12}) (R_{22}-{R_{22}}_{\text{opt},j-1})\\ R_{22}-{R_{22}}_{\text{opt},j-1}\end{array} \right)) + \sigma_{j}^2(\left(\begin{array}{c} -((UQ)_{11})^+ ((UQ)_{12}) {R_{22}}_{\text{opt},j-1}\\ {R_{22}}_{\text{opt},j-1} \end{array} \right)) \\ 
&= \sigma_s^2(\left(\begin{array}{c}X_{11}^+ X_{12}(R_{22}-{R_{22}}_{\text{opt}, j-1}) \\ R_{22}-{R_{22}}_{\text{opt}, j-1} \end{array}\right)).
\end{align*}
In particular, this establishes 
\[
\sigma_j^2(A-A_k) \leq \sigma_1^2(\left(\begin{array}{c}X_{11}^+ X_{12}(R_{22}-{R_{22}}_{\text{opt}, j-1}) \\ R_{22}-{R_{22}}_{\text{opt}, j-1} \end{array}\right)) \leq \|X_{11}^+ X_{12}(R_{22}-{R_{22}}_{\text{opt}, j-1})\|_2^2 + \|R_{22}-{R_{22}}_{\text{opt}, j-1} \|_2^2,
\]
and also by noting that the trailing $\min(m,n)-j$ singular values of $A-A_k$ are bound in this manner,
\[
\|A-A_k-(A-A_k)_{\text{opt},j-1}\|_F^2 \leq \|\left(\begin{array}{c}X_{11}^+ X_{12}(R_{22}-{R_{22}}_{\text{opt}, j-1}) \\ R_{22}-{R_{22}}_{\text{opt}, j-1} \end{array}\right) \|_F^2 = \|X_{11}^+ X_{12}(R_{22}-{R_{22}}_{\text{opt}, j-1})\|_F^2 + \|R_{22}-{R_{22}}_{\text{opt}, j-1} \|_F^2.
\]
This completes (\ref{eq:lu_frob})- (\ref{eq:lu_spectral2}). We proceed to the lower bound on $\sigma_i(A_k)$ claimed in (\ref{eq:assumption}).  If we let $\bar{A}_k$ for the moment denote the middle matrix in (\ref{eq:for_later}), then

\begin{align*}
\sigma_i(A_k) = \sigma_i(L'^{-1}\bar{A}_k R'^{-1}) &\geq \sigma_i( \left(\begin{array}{c} (L'^{-1}\bar{A}_k R'^{-1})_{11}\\ (L'^{-1}\bar{A}_k R'^{-1})_{21} \end{array}\right)) \\
&= \sigma_i( \left(\begin{array}{c} L_{11}'^{-1} \bar{A}_{11} R_{11}'^{-1}\\ \bar{A}_{21} R_{11}'^{-1} \end{array}\right)) \\
&= \sigma_i( \left(\begin{array}{c} L_{11}'^{-1}[ L_{11}'(U'Q)_{11} R_{11}] R_{11}'^{-1}\\ \bar{A}_{21} R_{11}'^{-1} \end{array}\right)) \\
&= \sigma_i( \left(\begin{array}{c} (U'Q)_{11} R_{11}R_{11}'^{-1} \\ {}[X_{21} X_{11}^+ X_{11} R_{11}] R_{11}'^{-1} \end{array}\right)) \numberthis   \label{eq:also_only_below} \\
&= \sigma_i( \left(\begin{array}{c} (U'Q)_{11} R_{11}R_{11}'^{-1} \\ {}[(U'Q)_{21} R_{11}]R_{11}'^{-1} \end{array}\right)) \\
&= \sigma_i(R_{11} R_{11}'^{-1}) .
\end{align*}
Here we have used the identities of Lemma \ref{lem:basic_facts}. For (\ref{eq:also_only_below}) in particular, we used $\bar{A}_{21} = (\bar{A}_{21} \bar{A}_{11}^+ ) (\bar{A}_{11})$ and then used (\ref{eq:lu1}) and (\ref{eq:lu3}) on the quantities in parentheses.
\end{proof}

\begin{remark}\label{rem:calculation}
Recall the sizes $V_1 = V[:,:l]$, $U_1 = U[:l', :]$.  When $l' = l$, the factorization in (\ref{eq:derivation}) can readily be rewritten in the more elegant form 
\begin{equation}
A_k = A V_1 (U_1 A V_1)^{-1} U_1 A \label{eq:simplified_square}
\end{equation}
One nice feature of this is that only $U_1, V_1$ are actually needed to compute $A_k$.  We will later see that the residual bounds in Proposition \ref{prop:LU} can be computed with only $U_1, V_1$ so it makes sense that we can find an analog of (\ref{eq:simplified_square}) for $l' > l$.  However, we actually need to set the rows $U_2 = U[l'+1:, :]$ to be a basis for the orthogonal complement of the rows of $U_1$ in order to achieve this.  Then $U^{-1} = [U_1^+, U_2^+]$, and we get a different form of (\ref{eq:derivation}) that is often faster to compute,
\begin{align*}
A_k &= U^{-1} \left(\begin{array}{c} I \\ \bar{A}_{21}\bar{A}^+_{11} \end{array}\right) \left (\begin{array}{cc} \bar{A}_{11} & \bar{A}_{12} \end{array}\right) V^{-1} \\
&= \left(\begin{array}{cc} U_1^+ & U_2^+ \end{array}\right) \left(\begin{array}{c} I \\ \bar{A}_{21}\bar{A}^+_{11} \end{array}\right) U_1 A \\
&=  (U_1^+ + U_2^+U_2A V_1 (U_1 A V_1)^+) U_1 A \\
&= \left[U_1^+ + (I-U_1^+U_1)A V_1 (U_1 A V_1)^+\right] \left[U_1 A\right] \\
&= [U_1^+ (I - (U_1 A V_1)(U_1 A V_1)^+) + (A V_1)(U_1 A V_1)^+ ] [ U_1 A] \numberthis \label{eq:simplified_rectangular}
\end{align*}
This final form should be viewed as a generalized LU-factorization.  The left factor is $m \times l'$ and the right factor ($U_1A$) is $l' \times n$.  Also recall $U_1$ is $l' \times m$ so the pseudo-inverse can be cheaply computed.
\end{remark}

We summarize the factorization discussed above in (only partially specified because of $U, V$ and the oversampling parameters $l, l'$) Algorithm \textbf{GLU} and Algorithm \textbf{RLU}. Recall that using square $U,V$ was only to help with the theoretical guarantees.  Therefore, in order to simplify notation, we now let $U,V$ denote what were up until now labeled as $U_1, V_1$.  We also emphasize Algorithm \textbf{RLU} is the special case of Algorithm \textbf{GLU} when the latter sets $l=l'$.

\begin{algorithm}
\protect\caption{$[T, S] =\textbf{GLU}(A,k)$.  Generalized LU approximation computes a low-rank approximation $A \approx A_k = T S$, where $T$ is a tall-skinny matrix and $S$ is a short-wide matrix.}
\begin{algorithmic}[1]
\label{alg:grlu}
\STATE \textbf{Input:} target rank $k$, matrix $A \in \mathbb{R}^{m \times n}$
\STATE \textbf{Output:} $T \in \mathbb{R}^{m \times l'}$, $S \in \mathbb{R}^{l' \times n}$, where
\STATE \textbf{Ensure:} $T = U^+ (I - (U A V)(U A V)^+) + (A V)(U A V)^+$, $S = U A$
\STATE Select oversampling parameters $l' \geq l \geq k$.
\STATE Generate full-rank $n \times l$ matrix $V$ and full-rank $l' \times m$ matrix $U$. 
\STATE $\hat{A} = UAV$
\STATE $T_1 = U^+(I-\hat{A}\hat{A}^+)$
\STATE $T_2 = AV$
\STATE $T_2 = T_2 \hat{A}^+$
\STATE $T = T_1 + T_2$
\STATE $S = UA$
\end{algorithmic}
\end{algorithm}

\begin{algorithm}
\protect\caption{$[T, \hat{A}, S] = \textbf{RLU}(A)$.  Rank-revealing LU computes a low-rank approximation $A \approx A_k = T \hat{A}^{-1} S$, where $T$ is a tall-skinny matrix, $S$ is a short-wide matrix, and $\hat{A}$ is a small dense matrix.} 
\begin{algorithmic}[1]
\label{alg:rlu}
\STATE \textbf{Input:} target rank $k$, matrix $A \in \mathbb{R}^{m \times n}$
\STATE \textbf{Output:} $T \in \mathbb{R}^{m \times l}$, $S \in \mathbb{R}^{l \times n}$, $\hat{A} \in \mathbb{R}^{l \times l} $
\STATE \textbf{Ensure:} $T = AV$, $S = U A$, $\hat{A} = UAV$
\STATE Select oversampling parameter $l \geq k$.
\STATE Generate a full-rank $n \times l$ matrix $V$ and a full-rank $l \times m$ matrix $U$. 
\STATE $T = A V$
\STATE $S = U A$
\STATE $\hat{A}= U T$
\end{algorithmic}
\end{algorithm}

The bounds in Proposition \ref{prop:LU} are not fully developed in that neither $R_{22}$ nor $R_{11}R_{11}'^{-1}$ have been examined yet, and also the choice of $U$ greatly influences $X$. In Section \ref{sec:QR} the $R_{22}$ factor will be studied; it will be bound in terms of $S^TV$ where $S$ is the right singular matrix of $A$.  See Proposition \ref{prop:rank_approx} and the resulting Theorem \ref{thm:LU}. Section \ref{sec:application} describes how choosing suitable random ensembles for $U, V$ allows for the Frobenius norm of the residual to be arbitrarily close to that of the truncated SVD, as well as many other bounds.  We therefore present what we consider to be our main results in Section \ref{sec:application}.  

\section{Relationship to other Approaches}\label{sec:relationship}
In this section we illustrate how \textbf{GLU} provides a general framework by proving the equivalence with Algorithm \textbf{PRR\_RLU} and Algorithm \textbf{RQR} below.  We will also see a close connection to the popular approach Algorithm \textbf{CW}, from Clarkson and Woodruff \cite{CW}.  We show show our approach is strictly more accurate when $l' > l$, and the same when $l'=l$.

\begin{algorithm}
\protect\caption{$[T, S] =\textbf{RQR}(A,k)$. Randomized QR approximation computes a low-rank approximation $A \approx A_k = TS$ where $T$ is a tall-skinny matrix with orthonormal columns, and $S$ is a short-wide matrix} 
\begin{algorithmic}[1]
\label{alg:rqr}
\STATE \textbf{Input:} target rank $k$, matrix $A \in \mathbb{R}^{m \times n}$
\STATE \textbf{Output:} orthogonal matrix $T \in \mathbb{R}^{m \times l}$, matrix $S \in l \times n$
\STATE \textbf{Ensure:} $T$ has orthonormal columns, $S = T^T A$
\STATE Select the oversampling parameter $l \geq k$.
\STATE Generate a full rank $n \times l$ matrix $V$. 
\STATE $\hat{A} = A V$.
\STATE $[T,\_] = \textbf{tQR}(\hat{A})$.
\STATE $S = T^T A$
\end{algorithmic}
\end{algorithm}

This version of the randomized SVD is described in section 5.2 of \cite{HMT}.  It additionally extracts rows from $A$ based on the product $AV$, leading to a speedup in many settings but at the cost of approximation quality.  See discussion in \cite{HMT} around (5.3).

\begin{algorithm}[H]
\protect\caption{$[T, S] =\textbf{PRR\_RLU}(A,k)$.  Randomized LU with row selection approximation based on Panel Rank-Revealing computes a randomized LU-factorization $A_k=T S$, performing sketching on the columns and a panel rank-revealing QR to select rows} 
\begin{algorithmic}[1]
\label{alg:prrrlu}
\STATE \textbf{Input:} target rank $k$, matrix $A \in \mathbb{R}^{m \times n}$
\STATE \textbf{Output:} $T \in \mathbb{R}^{m \times l}$ and $S \in \mathbb{R}^{l \times n}$
\STATE \textbf{Ensure:} $T = AV(P_1AV)^{-1} = Q (P_1Q)^{-1}$, $S = P_1 A \in \mathbb{R}^{l \times n}$ where $P$ is a permutation matrix and $P_1 = P[:l, :]$ and $Q \in \mathbb{R}^{m \times l}$ has orthonormal columns
\STATE Select oversampling parameter $l \geq k$
\STATE Generate a full-rank $n \times l$ random matrix $V$
\STATE $[Q, R] =\textbf{tQR}(A V)$.
\STATE Permutation $P$ is selected so that $PQ = \bar{Q} = \left(\begin{array}{c}\bar{Q}_{11} \\ \bar{Q}_{21} \end{array}\right)$ results in $||\bar{Q}_{21} \bar{Q}_{11}^{-1}||_{max}$ being bounded by a small constant (see \cite{MD}).  Here $P_1 = P[:l, :]$ and $\bar{Q}_{11} = P_1 Q$.
\STATE $T = P^T \left(\begin{array}{c}I \\ \bar{Q}_{21}\bar{Q}_{11}^{-1} \end{array}\right)$, also note then $T = AV(P_1AV)^{-1}$.
\STATE $S = P_1 A$
\end{algorithmic}
\end{algorithm}

We show how these algorithms fit into the LU-framework.  The fact is simple, but it appears to have been overlooked in the literature.  Therefore it has its own proposition:

\begin{proposition}\label{prop:equivalence}
$\textbf{PRR\_RLU}$ is equivalent to \textbf{RLU} when the latter chooses the same $V$ and $U :=P_1$.  

\vspace*{.3cm}
$\textbf{RQR}$ is equivalent to \textbf{RLU} when the latter chooses the same $V$ and $U := T^T$.
\end{proposition}

\begin{proof}
The proof is mainly to recall the various definitions.  First, Algorithm \textbf{PRR\_RLU} produces $A_k = Q (P_1Q)^{-1} P_1 A$.  As claimed within the algorithm, because $QR = AV$ it follows that $AV(P_1AV)^{-1} P_1 A = QR (P_1QR)^{-1}P_1 A = A_k$.  As $AV(P_1AV)^{-1} P_1 A $ is the output factorization of Algorithm \textbf{RLU}, the factorizations agree.

\vspace*{.5cm}
We move on to the other equivalence.  Recall $[T, R] = \textbf{tQR}(AV)$.  Selecting the same $V$ and $U = T^T$, the random LU-approximation given by Algorithm \textbf{RLU} would be 
\begin{equation}
\label{eq:transLUQR}
AV (T^T A V)^{-1} T^T A =  T R (T^T T R)^{-1} T^T A = T T^T A
\end{equation}
which agrees with Algorithm \textbf{RQR}.
\end{proof}

The most popular approach involving sketching from the left and right is perhaps the method first introduced in \cite{CW}, and also described in the overview \cite{W}.  It is not equivalent to $\textbf{GLU}$, but it is still closely related as we point out now.  As seen in Theorem 47 of \cite{W}, the output of what we call Algorithm \textbf{CW} after Clarkson, Woodruff is 
\begin{equation}
A \approx A_k' = AV_1(U_1AV_1)^+U_1A \label{eq:cw}
\end{equation}
where we take $U,V$ in the expanded form as below (\ref{eq:derivation}).  We now show that this procedure is strictly less accurate than \textbf{GLU} when $l' > l$, and the same when $l' = l$.

\begin{proposition}\label{prop:comparison}
Let $\bar{A} = UAV$ with $U = L' U'$ and $V = V' R'$ as in Proposition \ref{prop:LU}.  Additionally set $\tilde{A} = U_1 A$, and let $B$ be the projection of $\tilde{A}$ onto the orthogonal complement of $\tilde{A} V_1$.  Finally let $A_k$ be the output of Algorithm $\textbf{GLU}$ and $A_k'$ be the output (\ref{eq:cw}) of Algorithm $\textbf{CW}$.  Then
\begin{align*}
\|A-A_k'\|_F^2 &= \|A-A_k\|_F^2 + \|A_k-A_k'\|_F^2 \\
\|A_k-A_k'\|_F^2 &= \|U_1^+ B\|_F^2 \\ 
\|A-A_k'\|_2^2 &\leq \|A-A_k\|_2^2 + \|U_1^+ B\|_2^2
\end{align*}
 
\end{proposition}
\begin{proof}
Similar to Remark \ref{rem:calculation},  
\begin{align*}
A_k' = AV_1(U_1AV_1)^+U_1A &= U^{-1} \left(\begin{array}{c}\bar{A}_{11} \\ \bar{A}_{21} \end{array} \right)\bar{A}_{11}^+\left(\begin{array}{cc}\bar{A}_{11} & \bar{A}_{12} \end{array}\right) V^{-1} \\
&= U^{-1} \left( \begin{array}{cc} \bar{A}_{11} & \bar{A}_{11}\bar{A}_{11}^+\bar{A}_{12} \\ \bar{A}_{21} & \bar{A}_{22}-\mathscr{S}(\bar{A}_{11}) \end{array} \right) V^{-1}
\end{align*}
From this calculation and the calculation leading to (\ref{eq:to_compare}), it follows that 
\begin{align*}
\|A-A_k'\|_F^2 &= \|U^{-1} \left[ \bar{A}  - \left( \begin{array}{cc} \bar{A}_{11} & \bar{A}_{11}\bar{A}_{11}^+\bar{A}_{12} \\ \bar{A}_{21} & \bar{A}_{22}-\mathscr{S}(\bar{A}_{11}) \end{array} \right) \right] V^{-1}\|_F^2 \\
 &= \|\left( \begin{array}{cc} 0& U_1^+(I-\bar{A}_{11} \bar{A}_{11}^+)\bar{A}_{12} \\ 0 & \mathscr{S}(\bar{A}_{11}) \end{array} \right) \|_F^2 \\ 
 &= \|A-A_k\|_F^2 + \|U_1^+(I-\bar{A}_{11} \bar{A}_{11}^+)\bar{A}_{12}  \|_F^2 \\
 &= \|A-A_k\|_F^2 + \|U_1^+(I-\tilde{A} V_1 (\tilde{A} V_1)^+)\tilde{A} V_2  \|_F^2 \\
 &= \|A-A_k\|_F^2 + \|U_1^+(I-\tilde{A} V_1 (\tilde{A} V_1)^+)\left(\begin{array}{cc} \tilde{A} V_1' & \tilde{A}V_2' \end{array}\right) \|_F^2 \\
 &= \|A-A_k\|_F^2 + \|U_1^+(I-\tilde{A} V_1 (\tilde{A} V_1)^+) \tilde{A} \|_F^2 \\
 &= \|A-A_k\|_F^2 + \|U_1^+B\|_F^2
\end{align*}

This gives the Frobenius norm claims.  Compare this with the similar (\ref{eq:simplified_rectangular}).  Repeating the similar steps gives the spectral norm claim.  The only difference comes from an inequality instead of an equality in one step; for any unit vector $x$,
\begin{align*}
\|\left( \begin{array}{c} U_1^+(I-\bar{A}_{11} \bar{A}_{11}^+)\bar{A}_{12} \\ \mathscr{S}(\bar{A}_{11}) \end{array} \right)x \|_2^2 &= \|\left( \begin{array}{c} U_1^+(I-\bar{A}_{11} \bar{A}_{11}^+)\bar{A}_{12} x \\ \mathscr{S}(\bar{A}_{11}) x \end{array} \right) \|_2^2  \\ &=  \|U_1^+(I-\bar{A}_{11} \bar{A}_{11}^+)\bar{A}_{12} x \|_2^2 +  \|\mathscr{S}(\bar{A}_{11}) x \|_2^2  \\ 
&\leq \|\mathscr{S}(\bar{A}_{11})\|_2^2 + \|U_1^+(I-\bar{A}_{11} \bar{A}_{11}^+)\bar{A}_{12} \|_2^2
\end{align*}
\end{proof}
The work in \cite{CW} only considered the properties of the factorization (\ref{eq:cw}) in the context of Johnson-Lindenstrauss transforms, specifically when $l'$ is a poly-log factor larger than $l$, and not focusing on deterministic bounds.
\vspace*{.3cm}
If we compare the factorizations directly, perhaps the most obvious difference is that the output of Algorithm $\textbf{GLU}$ is typically rank $l'$ whereas the output of $\textbf{CW}$ is rank $l$.  Related to this, the factorization in $\textbf{CW}$ may be slightly cheaper to perform, although in typical settings ($k$ is relatively small) the same term dominants the cost of both algorithms.  Specializing to the SRHT ensemble for which our results are strongest, in Remark \ref{rem:comparison} we will see that the same bounds in Definition \ref{def:low_rank} and Definition \ref{def:spectral_approx} apply to both $\textbf{CW}$ and to $\textbf{GLU}$.  However, it does not appear the case that Theorem \ref{thm:srht} 's strong bound on Def. \ref{def:kernel_approx} can be carried over.

\section{QR Deterministic Bounds}\label{sec:QR}
The following lemma is important in randomized low rank approximation results. Our proof is novel, and (\ref{eq:tricky}), (\ref{eq:tricky2}) significantly generalize past versions.

\begin{proposition}\label{prop:rank_approx}
Let $A$ be an $m \times n$ matrix with $[P,\Sigma,S] = \textbf{SVD}(A)$.  As with Proposition \ref{prop:LU}, it is again convenient to suppose $V$ is $n \times n$ with $V=V'R'$ as described in (\ref{eq:restrict2}).  Also let $[Q, R] = \textbf{QR}(A V)$.  Then block $Q, R, S^TV, \Sigma$ using $Q_1 := Q[:, :l]$, $R_{11} = R[:l,:l]$, $(S^TV)_{11} = (S^TV)[:k,:l]$, and $\Sigma_1 = \Sigma[:k,:k]$, $\Sigma_2 = \Sigma[ k+1:,  k+1:]$.  Then the singular values of $Q_1Q_1^T A - A$ are identical to those of $R_{22}$, i.e. for any $1 \leq j \leq \min(m,n)-l$
\[
\sigma_j(Q_1Q_1^TA-A) = \sigma_j(R_{22}).
\]
Moreover, assuming $(S^TV)_{11}$ has full row-rank (and therefore $k \leq l$), we have that
\begin{align}
\|Q_1Q_1^TA-A\|_F^2 &\leq \|\Sigma_{2}\|_F^2 + \|\Sigma_2 (S^TV)_{21}(S^TV)_{11}^{+}\|_F^2 \label{eq:first_part} \\
\|Q_1Q_1^TA-A\|_2^2 &\leq \sigma_{k+1}^2 + \|\Sigma_2 (S^TV)_{21}(S^TV)_{11}^{+}\|_2^2 \label{eq:warmup_maybe}.
\end{align}
We may generalize this last equation with the goal of covering Definition \ref{def:kernel_approx}.  For any $1 \leq j \leq \min(m,n)-l$, there exists an $n \times (n-j+1)$ orthogonal matrix $\tilde{S}$ independent of $V$, satisfying
\begin{align} 
\sigma_j(Q_1Q_1^TA-A)_2^2 &\leq \sigma_{j+k}^2 + \|\Sigma_{j,2} (\tilde{S}^TV)_{21}(\tilde{S}^TV)_{11}^{+}\|_2^2  \label{eq:tricky} \\
\|(Q_1Q_1^TA-A)-(Q_1Q_1^TA-A)_{\text{opt},j-1}\|_F^2 &\leq \|\Sigma_{j,2}\|_F^2 + \|\Sigma_{j,2} (\tilde{S}^TV)_{21}(\tilde{S}^TV)_{11}^{+}\|_F^2 \label{eq:tricky2}
\end{align}
with $(\tilde{S}^TV)_{11}$ being $k \times l$ as before, and $\Sigma_{j,2} := \textbf{diag} (\sigma_{k+j}, \ldots, \sigma_{\min(m,n)}, 0, \ldots, 0  )$ is of dimension $(m-k) \times (n-k)$, where $\textbf{diag}$ denotes the diagonal matrix.
\end{proposition}

\begin{proof}
We first observe by direct computation,
\[
\sigma_j(Q_1Q_1^TA-A) =  \sigma_j(Q_2Q_2^TA) = \sigma_j(Q_2 \left(\begin{array}{cc} 0 & R_{22} \end{array} \right) R'^{-1} V'^T) = \sigma_j(R_{22}),
\]
to establish the first claim.  Next we invoke the common fact that for the spectral and Frobenius norms, $Q_1Q_1^TA$ is the best approximation to $A$ whose columns are in $\text{im}(Q_1)$.  For example, one can check that $Q_1Q_1^T$ satisfies the orthogonal projection properties with respect to these norms.  Set $\bar{A} = P^TA V$.  Then we explicitly propose an approximation $\tilde{A}_k$ whose columns are contained in $\text{im}(Q_1)= \text{im}(AV_1)$, namely
\[
\tilde{A}_k := P \left(\begin{array}{c} \bar{A}_{11} \\ \bar{A}_{12} \end{array} \right) \left(\begin{array}{cc} I & \bar{A}_{11}^+ \bar{A}_{12} \end{array} \right) V^{-1} = A V_1  \left(\begin{array}{cc} I & \bar{A}_{11}^+ \bar{A}_{12} \end{array} \right) V^{-1} 
\]
In contrast to before, $\bar{A}_{11}$ is $k \times l$, making it short and wide.  Repeating the algebra around (\ref{eq:to_compare}) in the first step,
\begin{align*}
\|A-\tilde{A}_k\|_F^2 = \|\mathscr{S}((P^T A V)_{11})\|_F^2 &= \|\mathscr{S}(\Sigma_1 (S^TV)_{11})\|_F^2 \\
&= \|\Sigma_2 (S^TV)_{22}- \Sigma_2 (S^TV)_{21} (\Sigma_1 (S^TV)_{11})^{+} \Sigma_1 (S^TV)_{12}\|_F^2 \\
&= \|\Sigma_2 (S^TV)_{22}- \Sigma_2 (S^TV)_{21} (S^TV)_{11}^{+}(S^TV)_{12}\|_F^2 \numberthis \label{eq:start_here} \\ 
&\leq \|\left(\begin{array}{cc}\Sigma_2 & - \Sigma_2 (S^TV)_{21} (S^TV)_{11}^{+} \end{array} \right)\|_F^2 \\
&= \|\Sigma_{2}\|_F^2 + \|\Sigma_2 (S^TV)_{21}(S^TV)_{11}^{+}\|_F^2 .
\end{align*}
To be clear, $S^TV$ was blocked so that $(S^TV)_{11}$ is $k \times l$.  Note we were able to distribute the pseudo-inverse in $(\Sigma_1 (S^TV)_{11})^+$.  In the generic case this follows from $(S^TV)_{11}$ having full row rank (this will be with probability 1 for suitably random $V$) and $\Sigma_1$ being invertible.  If $\Sigma_1$ has trailing 0 values, the assumption of full row rank of $(S^TV)_{11}$ ensures $\text{im}(AV_1) = \text{im}(A)$ and therefore we can instead use $\tilde{A}_k := A$ to get the bound of 0.

\vspace*{.3cm}
For the spectral norm bound, the steps are the same, except as in the proof in Proposition \ref{prop:comparison}, the final equality becomes an inequality.  

\vspace*{.3cm}
We actually are interested in the lower singular values as well though, so we extend the proof. In the following, $P_{Y}$ and $P_{AV_1}$ project onto the complements of the images of $Y$ and $AV_1$ respectively, $Y$ is rank $j-1$, and $AV_1$ is rank $l$.  The same projection notation applies to the other projections.  Using the additive Weyl inequality in the inequality step, similar to the use within Prop \ref{prop:LU}, for $s \leq \text{min}(m,n)-j$,

\begin{align*}
\sigma_{j+s-1}(Q_1Q_1^TA-A) = \sigma_{j+s-1}(P_{AV_1}A) &= \sigma_{j+s-1}(P_{AV_1}A-P_Y P_{AV_1}A+P_Y P_{AV_1}A)\\
&\leq \sigma_{s}(P_{Y} P_{AV_1}A) = \sigma_s(P_{Y + A V_1} A) = \sigma_{s}(P_{\tilde{Y}} P_{Y'} A)，
\end{align*}
Additionally, in the third equality, $Y+AV_1$ is used to refer to direct sum of the images of $Y$ and $AV_1$, and the equality holds under the assumption these spaces are orthogonal.  The fourth (last) equality holds when $\text{im}(Y') \oplus \text{im}(\tilde{Y}) = \text{im}(Y)\oplus \text{im}(AV_1)$, and $\text{im}(Y')$ is orthogonal to $\text{im}(\tilde{Y})$.

\vspace*{.3cm}
Now we make our choice of $Y+AV_1$.  First let $P_1=P[:,:j-1]$, the leading $j-1$ left singular vectors of $A$.  Noting that $\text{im}(P_1) \oplus \text{im}((P_1P_1^TA-A)V_1)$ is rank $l+j-1$ and contains $\text{im}(Q_1)=\text{im}(AV_1)$, and that $P_1$ is orthogonal to $(A-P_1P_1^TA)V_1$, these are valid choices of $Y'$ and $\tilde{Y}$ respectively.  In summary,
\begin{align*}
\text{im}(Y) \oplus \text{im}(AV_1) &= \text{im}(P_1) \oplus \text{im}((P_1P_1^TA-A)V_1) = \text{im}(P_1) \oplus \text{im}(AV_1) \numberthis \label{eq:direct_sums}\\
Y &= \text{ trailing j-1 columns of Q factor of } \text{tQR}(\left(\begin{array}{cc} AV_1 & P_1 \end{array}\right)) \\
Y' &= P_1 \\
\tilde{Y} &= P_{P_1} A V_1
\end{align*}
We emphasize in (\ref{eq:direct_sums}) that the first two are orthogonal direct sums.  This puts us essentially back into the situation surrounding (\ref{eq:start_here}).  Indeed,
\[
\sigma_{j+s-1}(Q_1Q_1^TA-A) \leq \sigma_{s}(P_{\tilde{Y}} P_{Y'} A) = \sigma_s(P_{BV_1} B) = \sigma_s(B-\tilde{Q}_1\tilde{Q}_1^T B),
\]
where $B = P_{P_1}A = A-P_1P_1^TA = A-A_{\text{opt},j-1}$, and $\tilde{Q}_1$ is an orthogonal matrix such that $\text{im}(\tilde{Q}_1) = \text{im}(B V_1) = \text{im}((A-A_{\text{opt},j-1}) V_1)$.  In particular with $s=1$,
\[
\sigma_j(Q_1Q_1^TA-A) \leq \sigma_{1}(\tilde{Q}_1\tilde{Q}_1^T(A-A_{\text{opt},j-1})-(A-A_{\text{opt},j-1})),
\]
as well as by comparing the singular values individually by varying $s$,
\[
\|(Q_1Q_1^TA-A)-(Q_1Q_1^TA-A)_{\text{opt},j-1}\|_F^2 \leq \|\tilde{Q}_1\tilde{Q}_1^T(A-A_{\text{opt},j-1})-(A-A_{\text{opt},j-1})\|_F^2.
\]
As a result, the RHS's we need to bound are the same as those bound when we established (\ref{eq:first_part}), and we may carry out the same steps as those around (\ref{eq:start_here}).  The only change is $A$ is replaced with $B = A-A_{\text{opt},j-1}$, and accordingly $Q_1$ changes to have the same image as $BV_1$.  The effect of this is the order of right singular vector matrix $S$ changes; the leading $j-1$ singular values and singular vectors removed.  To capture this change, we may notationally let $\tilde{S}$ be the permutation of the columns of $S$, moving the leading $j-1$ columns to the end. Then in the spectral case with $s=1$,
\[
\sigma_j^2(Q_1Q_1^TA-A)  \leq \sigma_{j+k}^2 + \|\Sigma_{j,2} (\tilde{S}^TV)_{21}(\tilde{S}^TV)_{11}^{+}\|_2^2 . 
\]
The Frobenius norm version follows similarly.
\end{proof}

In (\ref{eq:warmup_maybe}) and (\ref{eq:tricky}), one could factor out the $\sigma$ part to make the equations immediately take the form of Definitions \ref{def:low_rank} and \ref{def:kernel_approx}.  However, as in Theorem \ref{thm:srht}, the unfactored form can have advantages.

\begin{lemma}\label{lem:range}
Continue in the situation of Proposition \ref{prop:rank_approx}.  For $j \leq k$, 
\[
\sigma_j(Q_1 Q_1^T A) \geq \sigma_j(R_{11}R_{11}'^{-1}) \geq \sigma_{j}(A) \sigma_{min}((S^TV')_{11}))
\]
\end{lemma}

\begin{proof}
As in the previous proof, we see that the result is the same as if we right multiplied by $V'$ rather than $V$.  That is, we seek to bound from below the $j$-th singular value of
\[
Q_1^TA = \left(\begin{array}{cc} R_{11} & R_{12} \end{array} \right) R'^{-1} V'^T .
\]
Using this expression, we see that
\[
\sigma_j(Q_1Q_1^TA) \geq \sigma_j(R_{11} R_{11}'^{-1}) = \sigma_j(\Sigma (S^TV') [:, :l]) \geq \sigma_{j}(A) \sigma_{min}((S^TV')_{11}),
\]
where the reversed Weyl inequality was used in the last step.
\end{proof}
We state the following mainly to collect the results of the section into a single statement which resembles the definitions of strong QR factorizations from the literature.
\begin{proposition}\label{prop:QR}
Let $A$ be an $m \times n$ matrix with SVD $A = P \Sigma S^T$.  Set $[Q,R] = \textbf{QR}(A V)$, where $V=V'R'$ is an $n \times n$ matrix.  Then the singular values of $Q_1Q_1^T A - A$ are identical to those of $(m-l)\times(n-l)$ matrix $R_{22}$.  Moreover, 
\[
\|R_{22}\|_F^2 \leq \|\Sigma_2\|_F^2 + \|\Sigma_2 (S^TV)_{21} (S^T V)_{11}^+\|_F^2
\]
Also for $j \leq k$,
\begin{equation}
\sigma_{j}(A) \geq \sigma_{j}(Q_1 Q_1^T A) \geq \sigma_j(R_{11}R_{11}'^{-1}) \geq \sigma_{j}(A) \sigma_{\min}((S^TV')_{11}) \label{eq:overlooked}
\end{equation}
as well as for any given $j \leq \min(m,n)-k$, there is an orthogonal $n\times(n-j)$ matrix $\tilde{S}$ independent of $V$ such that
\begin{eqnarray} 
\sigma_j^2(R_{22}) &\leq& \sigma_{j+k}^2 + \|\Sigma_{j,2} (\tilde{S}^TV)_{21}(\tilde{S}^TV)_{11}^{+}\|_2^2  \\
\label{eq:QRoptSVD}
\|(R_{22})-(R_{22})_{\text{opt},j-1}\|_F^2 &\leq& \|\Sigma_{j,2}\|_F^2 + \|\Sigma_{j,2} (\tilde{S}^TV)_{21}(\tilde{S}^TV)_{11}^{+}\|_F^2,
\end{eqnarray}
with $(\tilde{S}^TV)_{11}$ being $k \times l$ as before, and $\Sigma_{j,2} := \textbf{diag} (\sigma_{k+j}, \ldots, \sigma_n, 0, \ldots, 0  )$ is of dimension $(m-k) \times (n-k)$, where $\textbf{diag}$ denotes the diagonal matrix.

\end{proposition}
\begin{proof}
Excluding the upper bound in (\ref{eq:overlooked}), the bounds are restatements of facts in Proposition \ref{prop:rank_approx} and Lemma \ref{lem:range}.  The upper bound is a consequence of the Weyl inequality,
\[
\sigma_j(Q_1 Q_1^T A) \leq \sigma_1(Q_1 Q_1^T) \sigma_j(A) = \sigma_j(A)
\]

\end{proof}

\section{Application of Randomness}\label{sec:application}
In this section, we combine our deterministic bounds with the past literature on sketching matrices.  There are three applications.  We first note that ensembles $U$ and $V$ used in Algorithm \textbf{GLU}'s guarantees in Proposition \ref{prop:LU} can be viewed through the oblivious subspace embedding property commonly used in literature.  Second, we specialize the random ensemble to the subsampled randomized Hadamard transform (SRHT) introduced in \cite{S} but whose analysis was strengthened in \cite{BG}.  Our approach fits nicely with their work to give particularly strong operator norm bounds, but in asymptotically less time. Third, we specialize to the Gaussian ensemble to see an application to analyzing the growth factor in Gaussian Elimination.

\vspace*{.3cm}
We begin by recalling Johnson-Lindenstrauss embeddings and oblivious subspace embeddings,

\begin{definition}\label{def:embedding}
An $(\epsilon, \delta, |T|)$ JL-transform (JLT) from $\mathbb{R}^n$ to $\mathbb{R}^{s}$ is a distribution $U \sim \mathbb{D}$ over $s \times n$ matrices.  It must with probability $1-\delta$ succeed in making 
\[
1-\epsilon \leq \|Ux\|_2^2 \leq 1+\epsilon
\]
hold for any given finite set $T \subset R^n$ of unit vectors.  We will assume $\epsilon < 1/6$.
\end{definition}

\begin{definition}\label{def:oblivious}
An $(k,\epsilon,\delta)$ oblivious subspace embedding (OSE) from $\mathbb{R}^n$ to $\mathbb{R}^s$ is a distribution $U \sim \mathbb{D}$ over $s \times n$ matrices.  It must with probability $1-\delta$ succeed in making 
\[
1-\epsilon \leq \sigma_{\min}(UQ) \leq \sigma_{\max}(U Q) \leq 1 + \epsilon
\]
hold for any given orthogonal $n \times k$ matrix $Q$.  We will assume $l \geq k$ and $\epsilon < 1/6$.
\end{definition}

The relationship between these definitions is interesting, and one good study is in \cite{BDN}. Clearly large enough $\delta, |T|$ in Definition \ref{def:embedding} will imply Definition \ref{def:oblivious} but sparse matrix ensembles find the former more challenging.  The discussion is beyond the scope of this work, however.  

\vspace*{.3cm}
The important consequence of Definition \ref{def:embedding} that we require is first found in \cite{S}.  Roughly speaking, a JLT can be used to approximate matrix products.  In \cite{W} Theorem 2.8 can be found the formulation of this that we will use:
\begin{lemma}\label{lem:embedding}
Let $U$ be drawn from an $(\epsilon, \delta, 1)$ JLT from $\mathbb{R}^n$ to $\mathbb{R}^s$, and let $A,B$ have $n$ rows.  Then with probability $1-\delta$,
\[
\|A^TU^TUB - A^TB \|_F \leq \epsilon \|A\|_F \|B\|_F .
\]
\end{lemma}

Similarly for Definition \ref{def:oblivious}, there is one consequence we require.  The first part is essentially Lemma 4.1 of \cite{BG} but we need to state it more generally.  The second part could be deduced with less work provided we strengthen the OSE assumption, but we prefer to avoid this by adding in Lemma \ref{lem:easy_facts}.
\begin{lemma}\label{lem:oblivious}
Let $U$ be an $s \times n$ matrix that is a $(k, \epsilon, \delta)$ OSE from $\mathbb{R}^n$ to $\mathbb{R}^s$, and $Q$ be an $(n \times k)$ orthogonal matrix.  Provided $\epsilon < 1/6$, then with probability $1-\delta$ both of the following hold,
\begin{eqnarray*}
\|(UQ)^+-(UQ)^T\|_2^2 \leq 3 \epsilon \\
\|U\|_2^2 = O\left(\frac{n}{k}\right),
\end{eqnarray*}
where in the second of these we require the additional assumption $\delta > 2 e^{-k/5}$.

\end{lemma}
\begin{proof}
Let $A = UQ$.  Then by Definition \ref{def:oblivious} the singular values of $A$ lie within $[1-\epsilon, 1+\epsilon]$ with probability $1-\delta$.  Let $l \times k$ diagonal matrix $\Sigma$ contain these singular vaues.  Therefore 
\begin{align*}
\|A^+-A^T\|_2 = \|\Sigma^T - \Sigma^+\|_2 &= \max\limits_{i \leq k} |\lambda_i - \lambda_i^{-1}| \\
&\leq |1-\epsilon - \frac{1}{1-\epsilon}| \leq | 1- \epsilon - (1 + 6/5 \epsilon)| \leq 3 \epsilon .
\end{align*}

For the second fact, let $V \leq \mathbb{R}^n$ be a uniformly distributed $k$-dimensional subspace with $\text{dim}(V)=k$ independent of $U$, i.e. $V$ is spanned by the first $k$ columns of a Haar distributed matrix on $\mathbb{R}^n$ independent of $U$.  A consequence of Definition \ref{def:oblivious} is that $\|Uv\|_2 \leq 2$ with probability $1-\delta$ holding uniformly for unit vectors $v$ contained in $V$. Otherwise some fixed subspace $V_0$ would also fail to have this property with probability $\delta$, violating Definition \ref{def:oblivious}.

\vspace*{.3cm}
Now let $x$ be the maximal right singular vector of $U$.  The subsequent Lemma \ref{lem:easy_facts} gives $\sup\limits_{v \in V, \|v\|_2=1} |\langle x, v \rangle | = \Omega(\sqrt{\frac{k}{n}})$ with probability $1-\delta$.  Next choose $v \in \text{argmax}_{v \in V, \|v\|_2=1}|\langle x, v \rangle |$ to be a unit-vector with smallest angle with respect to $x$, and observe $\|Uv\|_2 = \Omega(\sqrt{\frac{k}{n}}) \|Ux\|_2$.  We conclude $\|U\|_2^2 = O\left(\frac{n}{k}\right)$ with probability $1-\delta$. Otherwise this would contradict $\|Uv\|_2 \leq 2$ holding with probability $1-\delta$.
\end{proof}

\begin{lemma}\label{lem:easy_facts}
Let $V$ be a $k$-dimensional uniformly distributed subspace of $\mathbb{R}^n$, and $x \in \mathbb{R}^n$ be a unit vector drawn from a distribution independent of $V$. Then $\sup\limits_{v \in V, \|v\|_2=1} |\langle x, v \rangle | = \Omega(\sqrt{\frac{k}{n}})$ with probability $1 - 2 e^{-k/5}$.
\end{lemma}

\begin{proof}
We may assume $V= \text{span}(e_1, \dots, e_k)$, and represent $x$ as $\frac{(X_1, \dots, X_n)^T}{\sqrt{X_1^2 + \dots + X_n^2}}$ where $X_i$ are i.i.d. variance $\frac{1}{n}$ Gaussians.  Indeed, $V$ can be taken to be the first $k$ columns of Haar distributed orthogonal matrix $\tilde{V}$, and the WLOG assumption is equivalent to changing to the coordinates of $\tilde{V}$.  As a result, we are interested in 
\[
\sup\limits_{v \in V, \|v\|_2=1} |\langle x, v \rangle | = \frac{(X_1, \dots, X_n)}{\sqrt{X_1^2 + \dots + X_n^2}} \cdot \frac{(X_1, \dots X_k, 0, \dots)^T}{\sqrt{X_1^2 + \dots + X_k^2}} = \frac{\sqrt{X_1^2+ \dots + X_k^2}}{\sqrt{X_1^2 + \dots + X_n^2}}.
\]
Standard large-deviation bounds for chi-squared distribution, which is a sub-exponential random variable, can be used to lower bound this.  We take bounds from \cite{LM} (4.3), (4.4).  The right tail bound is
\[
\mathbb{P}[X_1^2 + \dots + X_n^2 > 1 + 2 \frac{\sqrt{\delta}}{\sqrt{n}} + 2\frac{\delta}{n} ] \leq e^{-\delta} ,
\]
and the left tail bound is
\[
\mathbb{P}[X_1^2 + \dots + X_k^2 < \frac{k}{n} - 2 \frac{ \sqrt{k \delta}}{n}] \leq e^{-\delta} .
\]
From these and setting $\delta = k/5$, we conclude 
\[
\sup\limits_{v \in V, \|v\|_2=1} |\langle x, v \rangle | \geq \left(\frac{\frac{k}{n} - 2 \frac{k}{\sqrt{5} n}}{1 + 2 \frac{\sqrt{k}}{\sqrt{5 n}} + 2\frac{k}{5 n} }\right)^{.5} \geq \frac{1}{25} \sqrt{\frac{k}{n}}
\]
holds with probability $1-2e^{-k/5}$.
\end{proof}
The following lemma largely follows the steps of \cite{BG} but is more general, employing the notions of JLT and OSE, and also treating the spectral norm.  It is a natural consequence of the prior lemmas, and will bridge the gap between deterministic Proposition \ref{prop:LU} and randomized Theorem \ref{thm:LU}.  We do not attempt to tightly bound the constant coefficients.

\begin{lemma}\label{lem:use_JL}
Assume $l\times m$ matrix $U$ is drawn from a distribution that is a $(k, \epsilon, \delta)$ OSE from $\mathbb{R}^m$ to $\mathbb{R}^l$. Let $B$ be a fixed $(m-k) \times n$ matrix, and $Q=[Q_1, Q_2]$ be a fixed orthogonal $m \times m$ matrix blocked so that $Q_1$ is $m \times k$.  Then provided $\delta > 2 e^{-k/5}$, with probability $1- \delta$
\[
\|(UQ_1)^+(UQ_2)A\|_2 = O\left(\frac{m}{k}\right) \, .
\]
Further assume $U$ is a $(\sqrt{\frac{\epsilon}{k}}, \delta, n)$ JLT, then with probability at least $1- 2\delta$,
\[
\|(UQ_1)^+ (UQ_2) B\|_F^2 = O\left(\epsilon\right)\| B\|_F^2 .
\]

\end{lemma}

\begin{proof}

For the Frobenius bound, apply Lemma \ref{lem:oblivious} in (\ref{eq:step1}), and Lemma \ref{lem:embedding} in (\ref{eq:step2}) by noting $Q_2^T Q_1 = 0$,
\begin{align*}
\|(UQ_1)^+ (UQ_2)B\|_F^2 &\leq 2 \|(UQ_1)^T (UQ_2) B\|_F^2 + 2 \|((UQ_1)^+ -(UQ_1)^T)(UQ_2)B\|_F^2 \\
&\leq 2 \|Q_1^TU^TUQ_2B\|_F^2 + 6 \epsilon \|UQ_2 B\|_F^2 \numberthis \label{eq:step1} \\
&\leq 2 \|Q_1^TU^TUQ_2B\|_F^2 + 12 \epsilon \| Q_2 B\|_F^2 \numberthis \label{eq:step3} \\
&\leq 2 \|Q_1^TU^TUQ_2B\|_F^2 + 12 \epsilon \|B\|_F^2 \\
&\leq 2 \frac{\epsilon}{k}\|Q_2 B\|_F^2 \|Q_1^T\|_F^2 + 12 \epsilon \|B\|_F^2 \numberthis \label{eq:step2} \\
&\leq 2 \epsilon \|B\|_F^2 + 12 \epsilon \|B\|_F^2 = 14 \epsilon \|B\|_F^2 .
\end{align*}
In the above, the step to (\ref{eq:step3}) used Definition \ref{def:embedding}, noting $\|Ux\|_2^2 \leq 2\|x\|_2^2$ holds for the $n$ columns of $Q_2B$ with probability $1-\delta$.

\vspace*{.3cm}
For the spectral bound, we may argue
\begin{align*}
\|(UQ_1)^+ (UQ_2)A\|_2^2 &\leq \|(UQ_1)^+ (UQ_2)\|_2^2 \cdot \|A\|_2^2 \\ 
&\leq \|(UQ_1)^T (UQ_2)\|_2^2 \cdot\|A\|_2^2 + \|((UQ_1)^+ -(UQ_1)^T)(UQ_2)\|_2^2 \cdot \|A\|_2^2 \\
&\leq \frac{7}{6}\|UQ_2\|_2^2 \cdot \|A\|_2^2 + 3 \epsilon \|(UQ_2)\|_2^2 \cdot \|A\|_2^2 \numberthis \label{eq:step4} \\
&= \frac{7}{6}\|U \|_2^2 \cdot \|Q_2\|_2^2 \cdot \|A\|_2^2 \numberthis \label{eq:step5}\\
&= O(\frac{m}{k}) \|A\|_2^2 \, .
\end{align*}
In the former steps, we note in particular that (\ref{eq:step4}) follows from Definition \ref{def:oblivious}, and \label{eq:step5} from Lemma \ref{lem:oblivious}.
\end{proof}

In the following, one of our main results, we continue with the notation of Propositions \ref{prop:LU} and \ref{prop:QR}.  We provide a bound on Definitions \ref{def:low_rank} and \ref{def:kernel_approx}.  While these bounds do appear quite weak (often weaker than a naive Frobenius norm adaptation), we note that they match the guarantees of past literature for algorithms running in $o(nmk)$ time, e.g. \cite{GCD}, \cite{HMT}, \cite{SSAA}.  On the other hand, in Theorem \ref{thm:srht} we notably achieve very sharp bounds on Definitions \ref{def:low_rank} and \ref{def:kernel_approx}, by exploiting the SRHT ensemble beyond its generic JLT and OSE properties.

\begin{theorem}\label{thm:LU}
Assume $U_1$ is drawn from a distribution that is an $(l, \epsilon, \delta)$ OSE from $\mathbb{R}^m$ into $\mathbb{R}^{l'}$.  Similarly assume $V_1^T$ is drawn from a distribution that is a $(k, \epsilon, \delta)$ OSE from $\mathbb{R}^n$ into $\mathbb{R}^l$.  Then provided $\delta > 2 e^{-k/5}$, with probability $1-2\delta$ for $j \leq k$,
\[
\sigma_j(A_k) = \Omega\left(\sqrt{\frac{k}{n}}\right) \sigma_{j}(A) \, .
  \]
 Fixing a given $1 \leq j \leq \min(m,n)-k$, with probability $1-4\delta$ we also have 
 \[
 \sigma_j(A-A_k) =  O\left(\sqrt{\frac{m n }{k l}}\right) \sigma_{k+j}(A).
 \]
If we additionally assume $U_1$ is drawn from a $(\sqrt{\frac{\epsilon}{l}}, \delta, m)$ JLT and similarly $V_1^T$ is drawn from a $(\sqrt{\frac{\epsilon}{k}}, \delta, n)$ JLT, then for a given $1 \leq j \leq \min(m,n)-k$,
\[ 
\|A-A_k\|_F^2 =\left(1+O(\epsilon)\right) \|A-A_{\text{opt},k}\|_F^2\, 
\]
holds with probability $1-4\delta$.
\end{theorem}
\begin{proof}
We start with the Frobenius norm bound.  The starting point is Proposition \ref{prop:QR}, which includes
\[
\|R_{22}-{R_{22}}_{\text{opt},j-1}\|_F^2 \leq \|\Sigma_{j,2}\|_F^2 + \|\Sigma_{j,2} (\tilde{S}^TV)_{21}(\tilde{S}^TV)_{11}^+\|_F^2 .
\]
Then as $V_1^T$ satisfies the JL properties, apply Lemma \ref{lem:use_JL} with $B = \Sigma_{j,2}^T$, $Q_1 = \tilde{S}_1$, $Q_2 = \tilde{S}_2$, and $U=V_1^T$, to conclude that for a given $1 \leq j \leq \min(m,n)-k$, $\|R_{22}-{R_{22}}_{\text{opt},j-1}\|_F^2 = (1+O(\epsilon))\|\Sigma_{j,2}\|_F^2$ with probability $1-2\delta$.  To complete the Frobenius bound, recall from Proposition \ref{prop:LU} that
\[
\|(A-A_k)-(A-A_k)_{\text{opt},j-1}\|_F^2 \leq \|R_{22}-{R_{22}}_{\text{opt},j-1}\|_F^2 +\|(UQ)_{11}^+(UQ)_{12}(R_{22} - {R_{22}}_{\text{opt},j-1})\|_F^2 ,
\]
and again apply Lemma \ref{lem:use_JL}, this time with $B = R_{22}-{R_{22}}_{\text{opt},j-1}$, to get $\|(A-A_k)-(A-A_k)_{\text{opt},j-1}\|_F^2 = (1 + O(\epsilon)) \|\Sigma_{j,2}\|_F^2$ with probability $1-4\delta$.  

\vspace*{.3cm}
The spectral bound proceeds similarly, but using the spectral bounds of Proposition \ref{prop:QR}, Proposition \ref{prop:LU}, and \ref{lem:use_JL} instead.  Thus 

\[
\|R_{22}-{R_{22}}_{\text{opt},j-1}\|_2^2 \leq \|\Sigma_{j,2}\|_2^2 + \|\Sigma_{j,2} (\tilde{S}^TV)_{21}(\tilde{S}^TV)_{11}^+\|_2^2 = O\left(\sqrt{\frac{m}{l'}}\right)\sigma_{j+k}^2 .
\]
And then using (\ref{eq:lu_spectral2}) of Proposition \ref{prop:LU},
\[
\sigma_j^2(A-A_k) \leq \|R_{22}-{R_{22}}_{\text{opt},j-1}\|_2^2 + \|(UQ)_{11}^+ (UQ)_{12}(R_{22}-{R_{22}}_{\text{opt},j-1})\|_2^2  = O\left(\frac{m n}{l' l}\right) \sigma_{j+k}^2,
\]
which proves the spectral claim.

\vspace*{.3cm}
For the multiplicative lower bound on the singular values of $A_{11}$, from (\ref{eq:assumption}) and  (\ref{eq:overlooked}) in Proposition \ref{prop:LU} and Proposition \ref{prop:QR} respectively, it follows that for $j \leq k$,

\[
\sigma_j(A_k) \geq \sigma_{j}(R_{11} R_{11}'^{-1}) \geq \sigma_{\min}((S^T_1V_1')) \sigma_{j}(A) = \Omega\left(\sqrt{\frac{k}{n}}\right) \sigma_j(A) .
\]
This last step requires additional explanation.  First,
\[
\sigma_{\min}(S^T_1V_1') = \sigma_{\min}(S^T_1V_1 {R'}_{11}^{-1}) \geq \sigma_{\min}(S^T_1V_1) \sigma_{\min}({R'}_{11}^{-1})) \geq  \frac{5}{6} \sigma_{\min}({R'}_{11}^{-1}) = \frac{5}{6} \frac{1}{\| {R'}_{11} \|_2} \, ,
\] 
where we used $\sigma_{\text{min}}(S^T_1 V_1) \geq \frac{5}{6}$ holds by Definition \ref{def:oblivious} with probability $1-\delta$.  It remains to upper bound $\|R_{11}'\|_2$.  We know $V_1 = V' \bigl( \begin{smallmatrix} R_{11}' \\ 0 \end{smallmatrix} \bigr)$, so $\|R_{11}'\|_2 = \|V_1\|_2$.  But $\|V_1\|_2 = O\left(\sqrt{\frac{n}{k}}\right)$ due to Lemma \ref{lem:oblivious}, with probability $1-\delta$.  This completes the proof of the lower bound.
\end{proof}

Next, we specialize to the SRHT ensemble in order to see a case where the bounds of Definition \ref{def:low_rank} and Definition \ref{def:kernel_approx} are stronger than in \ref{thm:LU}.

\begin{definition}\label{def:srht}
The SRHT ensemble embedding $\mathbb{R}^n$ into $\mathbb{R}^s$ is defined by generating $\sqrt{\frac{n}{s}}PHD$, where $P$ is $s \times n$ selecting $s$ rows, $H$ is the normalized Hadamard transform, and $D$ is a $n \times n$ diagonal matrix of uniformly random signs.  
\end{definition}

The key special additional property of the SRHT ensemble is from Lemma 4.8 of \cite{BG}.

\begin{lemma}\label{lem:from_bg}
Let $V^T$ be drawn from an SRHT of dimension $l \times n$.  Then for $m \times n$ matrix $A$ with rank $\rho$, with probability $1-2\delta$,
\[
\|AV\|_2^2 \leq 5 \|A\|_2^2 + \frac{\log(\rho/\delta)}{l}(\|A\|_F + \sqrt{8\log(n/\delta)} \|A\|_2)^2
\]
\end{lemma}

For context in understanding JLT and OSE properties, the SRHT with $10 \epsilon^{-1}(\sqrt{k} + \sqrt{8\log(m/\delta)})^2)\log(k / \delta)$ rows is an $(\epsilon, \delta, n)$ JLT for $\mathbb{R}^m$ as well as a $(k, \epsilon, \delta)$ OSE.  See \cite{BG} and \cite{S}, among other sources.  We may substitute these parameters into Theorem \ref{thm:LU}, but numerous other ensembles could also be used.  We have singled out the SRHT because it enjoys a remarkably good bound for the spectral norm approximation quality due to the prior lemma, but past work has not exploited this property fully.  In particular, when the spectral norm and Frobenius norm are comparable (i.e. quickly decaying singular values), the quality is constant in the dimension rather than polynomial. Loosely speaking, as long as $\frac{\|A-A_k\|_F}{\|A-A_k\|_2} = O(\sqrt{k})$, then $\|A-A_k\|_2$ is around a constant factor from that of the $k$-truncated SVD.  The theorem further strengthens this by proving the generalization to the lower singular values of $A-A_k$.

\begin{theorem}\label{thm:srht}
Let $U_1,V_1^T$ be drawn from SRHT ensembles with dimensions $l' \times m$, $n \times l$.  We set $l \geq 10 \epsilon^{-1}(\sqrt{k} + \sqrt{8\log(n/\delta)})^2\log(k / \delta)$ and $l' \geq 10\epsilon^{-1}(\sqrt{l} + \sqrt{8\log(m/\delta)})^2\log(k / \delta)$.  Letting $\rho$ be the rank of $A$, for simplicity assume $l' \geq \log(m/\delta)\log(\rho/\delta)$ and $l  \geq \log(n/\delta)\log(\rho/\delta)$.  Then for any fixed $1 \leq j \leq \min(m,n)-k$, with probability $1-5\delta$ the approximation of $A$ using \textbf{GLU}, $A_k$, satisfies
\begin{align*}
\sigma_j^2(A - A_k) &= O(1)\sigma_{k+j}^2 + O\left(\frac{\log(\rho / \delta)}{l}\right) \|A-A_{\text{opt},k+j-1}\|_F^2  \\ &= O\left(1+\frac{\epsilon \log(\min(m,n) / \delta)}{k \log(k/\delta)}\frac{\|A-A_{\text{opt},k+j-1}\|_F^2}{\sigma^2_{k+j}})\right) \sigma^2_{k+j} \, .
\end{align*}
\end{theorem}
\begin{proof}
It suffices to prove the first claim.  Begin by using Proposition \ref{prop:QR} and Lemma \ref{lem:oblivious},
\[
\sigma_j^2(R_{22}) \leq \| \Sigma_{j,2} \|_2^2 + \|\Sigma_{j,2} (\tilde{S}^T V)_{21}(\tilde{S}^T V)_{11}^+\|_2^2 \leq \| \Sigma_{j,2} \|_2^2 + 2\|\Sigma_{j,2} (\tilde{S}^T V)_{21}\|_2^2,
\]
with probability $1-\delta$.  Next apply Lemma \ref{lem:from_bg} to the second term to get
\begin{equation}
\label{eq:probboundR22}
\sigma_j^2(R_{22}) = O\left(1+\frac{\log(\rho/\delta)\log(n/\delta)}{l} \right) \|\Sigma_{j,2}\|_2^2 + O\left(\frac{\log(\rho/\delta)}{l}\right) \|\Sigma_{j,2}\|_F^2 = O(1)\|\Sigma_{j,2}\|_2^2 + O\left(\frac{\log(\rho/\delta)}{l}\right) \|\Sigma_{j,2}\|_F^2 ,
\end{equation}
where $\rho$ is the rank of $A$,
with probability $1-2\delta$.  Continue from the result of Proposition \ref{prop:LU},
\begin{align*}
\sigma_j^2(A-A_k) &\leq \|R_{22}-{R_{22}}_{\text{opt}, j-1}\|_2^2 + \|(U_1Q_1)^+ (U_1Q_2) (R_{22}-{R_{22}}_{\text{opt}, j-1})\|_2^2 \\
 & \leq \|R_{22}-{R_{22}}_{\text{opt}, j-1}\|_2^2 + 2 \|(U_1Q_2) (R_{22}-{R_{22}}_{\text{opt}, j-1})\|_2^2. 
\end{align*}
From Theorem \ref{thm:LU} we also know $\|(R_{22}-{R_{22}}_{\text{opt}, j-1})\|_F^2 \leq (1+O(\epsilon) )\|\Sigma_{j,2}\|_F^2 \leq 2\|\Sigma_{j,2}\|_F^2  $  because the SRHT with the parameter settings specified for $l$ and $l'$ satisfies the JLT and OSE properties. Thus repeating the same steps using Lemma \ref{lem:from_bg} and Lemma \ref{lem:oblivious} to complete the proof for the first bound,
\begin{align*}
\sigma_j^2(A-A_k) &\leq \|R_{22}-{R_{22}}_{\text{opt}, j-1}\|_2^2 + 2 \|(U_1 Q_2)(R_{22}-{R_{22}}_{\text{opt}, j-1}) \|_2^2 \\
&\leq O\left(1+\frac{\log(\rho/\delta)\log(m/\delta)}{l'} \right) \|R_{22}-{R_{22}}_{\text{opt}, j-1}\|_2^2 + O\left(\frac{\log(\rho/\delta)}{l'} \right) \|R_{22}-{R_{22}}_{\text{opt}, j-1}\|_F^2 \\
& =  O(1) \|R_{22}-{R_{22}}_{\text{opt}, j-1}\|_2^2 + O\left(\frac{\log(\rho/\delta)}{l'}\right) \|R_{22}-{R_{22}}_{\text{opt}, j-1}\|_F^2 \,.
\end{align*}
By using the bounds on $\sigma_j(R_{22})$ from \eqref{eq:probboundR22} and the fact that $\|R_{22}-{R_{22}}_{\text{opt}, j-1}\|_2 = \sigma_j(R_{22})$, we further obtain
\[
\sigma_j^2(A-A_k) \leq C_1 \sigma_{k+j}^2 + C_2 \frac{\log(\rho/\delta)}{l}\|\Sigma_{j,2}\|_F^2 .
\]
\end{proof}
A few remarks are in order.  
\begin{remark}
First, the SRHT ensemble is only defined for powers of 2.  This is not a theoretical issue because matrices can be padded.  However, as discussed in \cite{BG} there are orthogonal ensembles related to the SRHT, namely the discrete cosine transform and Hartley transform, for which the key probabilistic requirement in Lemma \ref{lem:from_bg} carries over, so this corollary also carries over.

\end{remark}
\begin{remark}
Second, we consider much of the work in this section as adapting \cite{BG} to algorithm \textbf{GLU} which sketches $A$'s columns and rows and proves a spectral norm bound comparable to the above.  Their work does not specify how to proceed after finding $A \approx Q_1 Q_1^T A$, and therefore follows \textbf{RQR}.  Therefore if one follows their approach, creating a compressed representation of $A$ would still require $O(nmk)$ time because $Q_1^T A$ must be computed.  We state the relevant part of their theorem here to provide context:

\begin{theorem}[\cite{BG}, Thm 2.1]
Let $A \in \mathbb{R}^{m\times n}$ have rank $\rho$ and $n$ a power of 2.  Fix an integer $k$ satisfying $2 \leq k < \rho$.  Let $0 < \epsilon < 1/3$ and $0 < \delta < 1$.  Let $Y = A V^T$ where $V\in \mathbb{R}^{r \times n}$ is drawn from the SRHT ensemble with $r = 6\epsilon^{-1}(\sqrt{k} + \sqrt{8\log(n/\delta)})^2)\log(k / \delta))$.  Then with probability $1-5\delta$
\[
\|A-Y Y^+A\|_2 \leq (4 + \sqrt{\frac{3 \log(n/\delta) \log(\rho/\delta)}{r}}) \|A-A_k\|_2 + \sqrt{\frac{3 \log(\rho / \delta)}{r}} \|A-A_k\|_F
\]
\end{theorem}

From this we see our Theorem \ref{thm:srht} has qualitatively the same accuracy guarantee on the residual error.  For many types matrices $A$, in particular for those with fast spectral decay, Theorem \ref{thm:srht} will be within a constant factor of the rank $k$ truncated-SVD's spectral approximation error.
\end{remark}

\begin{remark}\label{rem:comparison}
In comparing $\textbf{CW}$ with the outcomes of Theorems \ref{thm:LU} and Theorem \ref{thm:srht}, many results carry over, and we briefly sketch this here.  Given $U_1$ is from an SRHT ensemble, it is not difficult to see in Proposition \ref{prop:comparison} that $\frac{l}{\sqrt{m}}\tilde{A}$ has smaller singular values than $A$.  Moreover, the singular values of $B$ are bound through those of $\tilde{A}$, using Proposition \ref{prop:QR}, as $B$ is the projection of $\tilde{A}$ onto the sketch generated by $\tilde{A} V_1$.  Then $U_1^+$ is orthogonal besides undoing the scaling of $\tilde{A}$, by multiplying the singular values by $\frac{l}{\sqrt{m}}$. This sketch describes why the Frobenius norm and Spectral norm bounds on the residual still apply, i.e. Definition \ref{def:low_rank}.

\vspace*{.3cm}
The bound on Definition \ref{def:spectral_approx} we use is from Theorem \ref{thm:LU}, as it is not strengthened by using an SRHT ensemble. In particular it gave $\sigma_j(A_k) = \Omega(\sqrt{\frac{k}{n}}) \sigma_j(A)$. Recall that the deterministic identity behind the result is from \ref{prop:LU}, using the relations around (\ref{eq:also_only_below}). Intuitively because only the leading $l$ columns of $\bar{A}$ are used in proving this bound, the same argument applies.

\vspace*{.3cm}
In contrast to the above, Definition \ref{def:kernel_approx} does not fit very naturally with $\textbf{CW}$.  This is because, though $U_1^+B$  in Proposition \ref{prop:comparison} by itself can easily be bound, it does not necessarily interact nicely with $S(\bar{A}_{11})$.  Thus we are unable to extend the result of Theorem \ref{thm:srht} for $j > 1$.
\end{remark}
\begin{remark}\label{rem:runtime}
Let us consider the computational cost of computing the GLU approximation of $A$ through Theorem \ref{thm:srht}, storing the result in the form of (\ref{eq:introGLU}), following Algorithm \ref{alg:grlu}.  

\vspace*{.3cm}
Simply by following the algorithmic description, we see the largest cost terms are $O(nm \log(l') + mll')$.  We present a short table tabulating this.

\begin{table}[H]
\begin{tabular}{ll}
 $\hat{A} =U_1(AV_1)$ & $O(nm \log(l))$ \\
 $T_1 = U_1^+(I-\hat{A}\hat{A}^+)$ & $O(m l' \log(m)+l l'^2)$ because up to a factor $U_1$ has orthonormal columns, \\ & thus $U_1^+ = \sqrt{\frac{l'}{m}}(P H D)^T = \sqrt{\frac{l'}{m}} D H P^T$ \\ 
 $T_2 = AV_1$ & Stored from first step \\
 $T_2 = T_2 \hat{A}^+$ & $O(mll')$ \\
 $T = T_1 + T_2$ & $O(ml')$\\
 $S = U_1A$ & $O(m n \log(l'))$
\end{tabular}
\end{table}

\vspace*{.3cm}
Specializing as in the theorem, we additionally required $l \geq 10 \epsilon^{-1}(\sqrt{k} + \sqrt{8\log(n/\delta)})^2\log(k / \delta)$ and $l' \geq 10\epsilon^{-1}(\sqrt{l} + \sqrt{8\log(m/\delta)})^2\log(k / \delta)$.  Using these bounds on $l$ and $l'$, we say the runtime is $\tilde{O}(nm + k^2 m \epsilon^{-3})$.  Various poly-log factors are hidden here, involving $n, m, k, \delta$.  In more detail, plugging in $l$ and $l'$ into the prior complexity bound and assuming $m < n$ so that $l' = O(l \epsilon^{-1} \log(k/\delta))$, we get Big-Oh of 
\[
 nm \log\left(\epsilon^{-2}(k +\log(n / \delta)) \log^2(k/\delta)\right) + m \epsilon^{-3}(k^2 +\log^2(n / \delta))\log^3(k/\delta)).
\]
Note that in the runtime bound, because there is asymmetry between $m$ and $n$, it turns out to be faster if $m < n$ and thus $A$ is short-wide. If this is not the case for $A$, then one could simply run the algorithm on $A^T$.
\end{remark}

\begin{remark}
As stated, Theorem \ref{thm:srht} provides bounds for the \textbf{GLU} with sketching from the left and right.  We noted in the prior remark how this retains the performance of \cite{BG} while increasing the speed.  We could stop the analysis at (\ref{eq:probboundR22}), and also borrow the bounds already found in Theorem \ref{thm:LU} and Proposition \ref{prop:QR}.  Then we obtain new bounds for the randomized QR factorization,
\begin{corollary}\label{cor:srht_qr}
Let $n \times l$ matrix $V_1^T$ be drawn from an SHRT ensemble, $l \geq 10 \epsilon^{-1}(\sqrt{k} + \sqrt{8\log(n/\delta)})^2\log(k / \delta)$, and for simplicity assume $l  \geq \log(n/\delta)\log(\rho/\delta)$.  Then we have 
\[
\|R_{22}-{R_{22}}_{opt,j-1}\|_F^2 \leq (1+O(\epsilon))\|A-A_{\text{opt},k+j-1}\|_F^2\, ,
\]
with probability $1-2\delta$, as well as
\[
\sigma_j^2(R_{22}) \leq O(\sigma_{k+j}^2) + O(\frac{\log(\rho/\delta)}{l})\|A-A_{\text{opt},k+j-1}\|_F^2 \, ,
\] 
for $1 \leq j \leq \text{min}(m,n)-k$ with probability $1-3\delta$ for a particular $j$.  We also have upper and lower bounds on the largest singular values, as for $1 \leq j \leq k$,
\[
\sigma_j(A) \geq \sigma_j(Q_1Q_1^TA) = \Omega(\sqrt{\frac{k}{n}}) \sigma_{j}(A) 
\]
holds with probability $1-2\max(\delta, e^{-k/5})$.  Actually, borrowing the deterministic bound of \cite{G} found in equation (4.7), 
\[
\sigma_j(A) \geq \sigma_j(Q_1Q_1^TA) \geq \frac{\sigma_{j}}{1+O(\sqrt{\frac{n}{k}}) \frac{\sigma_{k+1}}{\sigma_j}}
\]
holds with probability $1-\delta$.
\end{corollary}
\end{remark}

We move on to the third application, controlling the growth factor during Gaussian elimination by right and left multiplication by square random matrices.  The theoretical result we establish is that the growth factor is well behaved if we multiply by square Gaussian random matrices.  Note the bounds in Propositions \ref{prop:LU} and \ref{prop:rank_approx} will in this case be the same for Gaussian random matrices as for Haar random matrices, because they differ by lower and upper triangular factors and $U_1, V_1$ are now square.  We make use of bounds proven for the Haar ensemble.  The work \cite{DDH}, which viewed the problem in terms of the Haar ensemble, required a randomized QR-factorization as a subroutine to compute the generalized Schur-decomposition of the matrix by a divide-and-conquer approach.  This required a bound on the smallest singular value of the $k \times k$ minors.  Eventually a tight bound on these was given in \cite{D} by means of the exact probability distribution, which we will use.

\vspace*{.3cm}
As pointed out in \cite{BDDR}, Theorem 3.2 and Lemma 3.5 of \cite{D} give an exact density of the smallest singular value of a Haar minor.  Analyzing this formula gives the following bound, which is sharp up to a constant in the primary range of interest, $\sigma_{\min} = O\left(\frac{1}{\sqrt{k(n-k)}}\right)$.
\begin{lemma} 
\label{lem:low_bd}
Let $\delta>0$, $k, (n-k)>30$; then $\mathbb{P}\left [ \sigma_{\min} \leq \frac{\delta}{\sqrt{k(n-k)}} \right] \leq  2.02 \delta$.
\end{lemma}

We will define the $\ell_2$ growth factors of $\bar{A}$ as $\rho_U(\bar{A}) := \max\limits_p \|\mathscr{S}_p(\bar{A})\|_2 / \|\bar{A}\|_2$ and $\rho_L(\bar{A}) := \max\limits_p \|\bar{A}_{21}\bar{A}_{11}^{-1}\|_2$ where $\mathscr{S}_p$ is the Schur complement of the top $p \times p$ block.  From Proposition \ref{prop:LU},(\ref{eq:lu1}), and Proposition \ref{prop:rank_approx} it is not difficult to see that both are bounded as 

\begin{align*}
\rho_U(\bar{A}), \rho_L(\bar{A}) &\leq \max\limits_{j} \left[ \| X[:j,:j]^{-1}\|_2 \|R[j+1:,j+1:]\|_2 / \|\bar{A}\|_2 \right] \\
& \leq \max\limits_{j} \left[ \| (UQ)[:j,:j]^{-1}\|_2 \|(S^TV)[j+1:,j+1:]^{-1}\|_2 \right] \\
& = \max\limits_{j} \left[ \sigma_{\min}^{-1}((UQ)[:j,:j]) \sigma_{\min}^{-1}((S^TV)[:j,:j]) \right] .
\end{align*}
Note that $\rho_U$ and $\rho_L$ control what is typically called the growth factor of $\bar{A}$.  The growth factor is the largest magnitude entry appearing in the matrices $L,U$ returned by Gaussian Elimination.  This is because of norm equivalence, with the operator and max-element norm differing by at most a factor of $\sqrt{n}$.  Therefore our $\ell_2$ growth factors are equivalent for the purpose of proving stability.

\begin{corollary}\label{cor:provable}
Suppose we want to solve $Ax = b$ by Gaussian Elimination, and we precondition, postcondition $A$ by Haar distributed matrices $U,V$.  That is, we solve $UAVx' = Ub$ and output $V^T x'$.  Then the $U$ and $ L$ $\ell_2$-growth factors introduced above satisfy
\[
\mathbb{E}[\log(\max(\rho_U(\bar{A}), \rho_L(\bar{A})))] = O(\log(n))
\]
\end{corollary}

\begin{proof}
Because $U$ and $V$ are Haar, the matrices $UQ$ and $S^TV$ in Propositions \ref{prop:LU} and \ref{prop:rank_approx} are Haar distributed.  Apply Lemma \ref{lem:low_bd} to the minors (call them generically $M$) of $UQ$ and $S^TV$ with size in the range $[30, n-30]$, 
\[
\mathbb{P}[\sigma_{\min}^{-1}(M) > n^{2+a}] < 2.02n^{-1-a}
\]
To control all minors in this range, simply perform a union bound over all $<2n$ minors being considered.  Let $B_1$ be the inverse of the smallest singular value of the minors in range $[30, n-30]$ of $UQ$ and $S^TV$.  Then $\mathbb{P}[B_1 \geq n^{2+a}] \leq 4.04 n^{-a}$.  Setting $a = x-2$, this is $\mathbb{P}[\log_n(B_1) \geq x] \leq 4.04 n^{2-x}$.

\vspace*{.3cm}
To deal with the minors in range $[0, 30]$, we cite a result in random matrix theory which says that these minors scaled by $\sqrt{n}$ approach a matrix of i.i.d. $N(0,1)$ random variables.  The convergence is with respect to total variation distance, see \cite{J}.  Let $B_2$ be the inverse of the smallest singular value of these 60 minors.  For the claimed result, what matters is $\mathbb{E}[\log_n(B_2)] = C_1'$ for some constant $C_1'$, due to the $\frac{1}{\sqrt{n}}$ scaling.  This is apparent from work similar to \cite{D} but for Gaussian matrices, see for example the bound on the condition number in \cite{CD}.

\vspace*{.3cm}
Combining the bounds for $B_1$ and $B_2$,
\begin{align*}
\mathbb{E}[\log_n(\max(\rho_U(\bar{A}), \rho_L(\bar{A})))]  &\leq \mathbb{E} [ \log_n (\max\limits_j \left [\sigma_{\min}^{-1}((UQ)[:j,:j]) \sigma_{\min}^{-1}((S^TV)[:j,:j]))) \right] \\ 
&\leq \mathbb{E}[\log_n(B_1)] + \mathbb{E}[\log_n(B_2)] \\
&\leq C_1' + \int_0^2 1dx + 4.04 \int_{2}^{\infty} n^{2-x} dx \\
&= C_1 + 4.04 \log(n) \int_{0}^{\infty} e^{-x} dx \\
&\leq C \log(n)
\end{align*}
\end{proof}
Of course, it is impractical to use a Gaussian or Haar matrix to condition a matrix in this context.  We might as well then solve the system by means of QR-factorization.  However, this sheds light on the strategy of using conditioners to avoid pivoting during Gaussian Elimination.  This has been popularized in work such as \cite{B}.  The theoretical support of such work has been lacking.  Corollary \ref{cor:provable} is the first theoretical result we are aware of that shows a random conditioners can be used to provably avoid the need to pivot.  

\vspace*{.3cm}
It also could be considered a generalization of the well-known fact that Gaussian random matrices have low pivot growth during Gaussian elimination.  Indeed, we have shown that this is the case for any distribution of singular values not just that of the Gaussian random matrix.  The most interesting question still remains if faster conditioners can be used to make the approach both theoretically and practically sound for all matrices $A$.  More concretely we pose the question,

\begin{remark}\label{rem:open}
Is there a random matrix ensemble $S$ such that $S A$ can be computed quickly, but also $\sigma_{\min}((S A)[:k, :k]) = O\left(\frac{1}{\text{poly}(n)}\right)$ when $A$ is an orthogonal matrix?
\end{remark}

\section{Conclusion}
We have provided a thorough analysis of a new low-rank approximation procedure \textbf{GLU}.  Along the way, we have seen it is closely related to many different past approaches.  Our procedure is as fast as past approaches to within a log factor, and comes with spectral and frobenius norm bounds on the residual, as well as multiplicative bounds for the other singular values.

\vspace*{.3cm}
For future work, Remark \ref{rem:open} seems useful and interesting. Finding applications which particularly benefit from the speed and accuracy guarantees of our procedure is also of interest.

\bibliographystyle{plain}
\bibliography{references}

\end{document}